\documentclass[12pt,british,refpage,intoc,bibliography=totoc,index=totoc,BCOR=7.5mm,captions=tableheading]{extarticle}

\usepackage[T1]{fontenc}
\usepackage[latin9]{inputenc}
\usepackage[a4paper]{geometry}
\usepackage{color}
\usepackage{babel}
\usepackage{amsmath}
\usepackage{amsthm}
\usepackage{amssymb}
\usepackage[numbers]{natbib}
\usepackage[unicode=true,pdfusetitle,
 bookmarks=true,bookmarksnumbered=true,bookmarksopen=false,
 breaklinks=false,pdfborder={0 0 0},pdfborderstyle={},backref=false,colorlinks=true]
 {hyperref}
\hypersetup{
 linkcolor=black, citecolor=blue, urlcolor=black, filecolor=blue, pdfpagelayout=OneColumn, pdfnewwindow=true, pdfstartview=XYZ, plainpages=false}
\usepackage{breakurl}

\makeatletter
\numberwithin{equation}{section}
\numberwithin{figure}{section}
\theoremstyle{plain}
\newtheorem{thm}{\protect\theoremname}[section]
  \theoremstyle{remark}
  \newtheorem{rem}[thm]{\protect\remarkname}
  \theoremstyle{definition}
  \newtheorem{defn}[thm]{\protect\definitionname}
  \theoremstyle{definition}
  \newtheorem{example}[thm]{\protect\examplename}
  \theoremstyle{plain}
  \newtheorem{cor}[thm]{\protect\corollaryname}
  \theoremstyle{plain}
  \newtheorem{lem}[thm]{\protect\lemmaname}
  \theoremstyle{plain}
  \newtheorem{prop}[thm]{\protect\propositionname}
  \theoremstyle{plain}
  \newtheorem*{lem*}{\protect\lemmaname}

\@ifundefined{date}{}{\date{}}
\usepackage{caption}
\usepackage[nottoc]{tocbibind}

\makeatother

  \providecommand{\corollaryname}{Corollary}
  \providecommand{\definitionname}{Definition}
  \providecommand{\examplename}{Example}
  \providecommand{\lemmaname}{Lemma}
  \providecommand{\propositionname}{Proposition}
  \providecommand{\remarkname}{Remark}
\providecommand{\theoremname}{Theorem}

\begin{document}

\title{Parabolic type equations associated with the Dirichlet form on the
Sierpinski gasket}

\author{Xuan~Liu\thanks{Nomura International, 30/FL Two International Finance Centre, Hong
Kong. Email: \protect\href{mailto:chamonixliu@163.com}{chamonixliu@163.com}.\protect \\
This research was carried out when this author was reading DPhil in
Mathematics at the University of Oxford.}\hspace{0.7em}and Zhongmin~Qian\thanks{Research supported partly by an ERC grant (Grant Agreement No. 291244
ESig). Mathematical Institute, University of Oxford, Oxford, OX2 6GG,
United Kingdom. Email: \protect\href{mailto:zhongmin.qian@maths.ox.ac.uk}{zhongmin.qian@maths.ox.ac.uk}}}
\maketitle
\begin{abstract}
By using analytic tools from stochastic analysis, we initiate a study
of some non-linear parabolic equations on Sierpinski gasket, motivated
by modellings of fluid flows along fractals (which can be considered
as models of simplified rough porous media). Unlike the regular space
case, such parabolic type equations involving non-linear convection
terms must take a different form, due to the fact that convection
terms must be singular to the ``linear part'' which defines the
heat semigroup. In order to study these parabolic type equations,
a new kind of Sobolev inequalities for the Dirichlet form on the gasket
will be established. These Sobolev inequalities, which are interesting
on their own and in contrast to the case of Euclidean spaces, involve
two $L^{p}$ norms with respect to two mutually singular measures.
By examining properties of singular convolutions of the associated
heat semigroup, we derive the space-time regularity of solutions to
these parabolic equations under a few technical conditions. The Burgers
equations on the Sierpinski gasket are also studied, for which a maximum
principle for solutions is derived using techniques from backward
stochastic differential equations, and the existence, uniqueness,
and regularity of its solutions are obtained.
\end{abstract}
\begin{center}
\emph{\small{}}%
\begin{minipage}[t]{0.66\paperwidth}%
\textbf{\small{}Keywords}{\small{}\enskip{}Brownian motion, Dirichlet
forms, Sierpinski gasket,}\textbf{ }{\small{}Sobolev inequalities,
semi-linear parabolic equations\medskip{}
}{\small \par}

\textbf{\small{}Mathematics Subject Classification}{\small{} }\textbf{\small{}(2000)}{\small{}\enskip{}28A80,
60J45}{\small \par}%
\end{minipage}
\par\end{center}{\small \par}

\section{\label{sec:-2}Introduction}

The analysis on fractals has attracted attentions of researchers in
the last decades, not only for the reason that fractals are archetypal
examples of spaces without suitable smooth structure, but also because
fractals are examples of interesting models in statistical mechanics.
Many objects in nature (e.g. percolation clusters in disordered media,
complex biology systems, polymeric materials, and etc.) possess features
of fractals (see e.g. \citep{Man77} for details). Fractals appear
as scaling limits of lattices. Lattice models (e.g. the Ising models
and their variants) have been extensively studied in statistical mechanics,
and properties for scaling limits have been derived using conformal
field theory in dimension two. 

Since a calculus on fractals is not available, the theory of Dirichlet
forms on measure-metric spaces and stochastic calculus are the analytic
tools employed for the study of analysis problems on fractals, and
many interesting results have been established in the past decades.

Early works on analysis on fractals however have been focused mainly
on diffusion processes and the corresponding Dirichlet forms (see
e.g. \citep{BP88,Ki89,Ham92,Ki93,KL93,FHK94,BH97,BK97,Ham97,Ki01}
and etc.). Brownian motion on the Sierpinski gasket was first constructed
by S.~Goldstein and S.~Kusuoka as the limit of a sequence of (scaled)
random walks on lattices (cf. \citep{Gold87,Ku87}). J.~Kigami \citep{Ki89}
has obtained an analytic construction of the Dirichlet form via finite
difference schemes. The construction of gradients of functions with
finite energy has been given in S.~Kusuoka in \citep{Ku89}, where
a significant difference between Euclidean spaces and fractals has
also been revealed (see \citep[Section 6]{Ku89}). On the Sierpinski
gasket for example, volumes of sets and energies of functions are
measured in terms of two mutually singular measures, the Hausdorff
measure and Kusuoka's measure (see Section \ref{sec:-3} below for
definitions). By virtue of the results obtained in \citep{Ku89},
gradients of functions on the Sierpinski gasket may be defined as
square integrable functions with respect to Kusuoka's measure (cf.
Section \ref{sec:-3}). Roughly speaking, the gradient of a function
with finite energy is the square root of the density of its energy
measure with respect to Kusuoka's measure. There have been interests
in the understanding of gradients of functions and non-linear partial
differential equations on fractals with non-linearities involving
first-order derivatives (see e.g. \citep{Tep00,IRT12,HT13,HRT13,HT15}
and references therein). A new class of semi-linear parabolic equations
involving singular measures on the Sierpinski gasket was proposed
and studied in \citep{LQ16}, where, among other things, a Feynman\textendash Kac
representation was obtained assuming the existence of weak solutions.

In the present paper, we establish the existence and uniqueness of
solutions to the semi-linear parabolic PDEs proposed in \citep{LQ16},
and derive the regularity of solutions. A crucial ingredient in our
argument is a new type of Sobolev inequalities on the Sierpinski gasket
(and the infinite gasket) involving different measures (which can
be mutually singular). To author's knowledge, this type of Sobolev
inequalities on fractals has not been investigated before, and is
of mathematical interests on its own. We formulate and study the Burgers
equations on the gasket, which is an archetype of non-linear PDEs
with non-Lipschitz coefficients, and also as a simplified model of
flows in porous medium. The difficulty in our case is that there exists
no suitable analogue of the Cole-Hopf transformation on the gasket.
Instead we tackle the problem by using a Feynman\textendash Kac representation
and an iteration argument.

This paper is organized as follows. We introduce in Section \ref{sec:-3}
the notations and definitions which will be effective throughout the
paper. Several preliminary results are also reviewed in the same section.
In Section \ref{sec:}, we give the formulation and the proof of new
Sobolev inequalities on the Sierpinski gasket (and the infinite gasket),
which will be needed in latter sections. The optimal exponents and
a sufficient and necessary condition for the validity of these inequalities
are also given in this section. Section \ref{sec:-1} is devoted to
the semi-linear parabolic PDEs on the gasket, where we establish the
existence and uniqueness and the regularity of solutions. In Section
\ref{sec:-4}, we apply the results in previous sections to the study
of the Burgers equations on the gasket, which are the analogues of
the Burgers equations on $\mathbb{R}$.

The results of this paper are presented only for the Sierpinski gasket
in $\mathbb{R}^{2}$, we however believe that our results also hold
for Sierpinski gaskets in higher dimensions. The main results and
the arguments given in this paper can be adapted accordingly without
difficulties.

\section{\label{sec:-3}Preliminaries}

In this section, we set up several notations and definitions which
will be in force throughout this paper.

\subsubsection*{Sierpinski gaskets}

Let $\mathbf{F}_{i}:\mathbb{R}^{2}\to\mathbb{R}^{2},\;i=1,2,3$ be
the contractions defined by $\mathbf{F}_{1}(x)=2^{-1}x,\;\mathbf{F}_{2}(x)=2^{-1}[x+(1,0)],\;\mathbf{F}_{3}(x)=2^{-1}[x+(1/2,\sqrt{3}/2)],\;x\in\mathbb{R}^{2}$.
Let $\mathrm{V}_{0}=\big\{(0,0),\,(1,0),\,(1/2,\sqrt{3}/2)\big\}$.
Define $\mathrm{V}_{m},\,m\in\mathbb{N}_{+}$ inductively by $\mathrm{V}_{m}=\bigcup_{i=1,2,3}\mathbf{F}_{i}(\mathrm{V}_{m-1})$.
Let $\hat{\mathrm{V}}_{m}=\bigcup_{k=0}^{\infty}2^{k}\,[\mathrm{V}_{m+k}\cup(-\mathrm{V}_{m+k})],\;m\in\mathbb{N}$.
The \emph{(compact) Sierpinski gasket} $\mathbb{S}$ and\emph{ }the
\emph{infinite Sierpinski gasket} $\hat{\mathbb{S}}$ are defined
to be the closures $\mathbb{S}=\mathrm{cl}\,(\bigcup_{m=0}^{\infty}\mathrm{V}_{m})$
and $\hat{\mathbb{S}}=\mathrm{cl}\,(\bigcup_{m=0}^{\infty}\hat{\mathrm{V}}_{m})$
respectively. $\hat{\mathbb{S}}$ can be written as a countable union
$\hat{\mathbb{S}}=\bigcup_{i\in\mathbb{Z}}\tau_{i}(\mathbb{S})$,
where $\tau_{i}:\mathbb{R}^{2}\to\mathbb{R}^{2},\;i\in\mathbb{Z}$
are translations of $\mathbb{R}^{2}$ such that $\tau_{i}(\mathbb{S}),\,i\in\mathbb{Z}$
have non-overlapping interiors. To our purpose, the labelling of the
translations $\tau_{i},\;i\in\mathbb{Z}$ is immaterial. We should
point out that there are many different infinite versions of $\mathbb{S}$
(see e.g. \citep[Section 5]{Tep98}). The $\hat{\mathbb{S}}$ we use
in the present paper is only one of them.

Let $\mathrm{W}_{\ast}=\{\omega=\omega_{1}\omega_{2}\omega_{3}\mathord{\dots}:\omega_{i}\in\{1,2,3\},\;i\in\mathbb{N}_{+}\}$
the set of infinite ordered sequences $\omega$ of symbols in $\{1,2,3\}$.
For each $\omega=\omega_{1}\omega_{2}\omega_{3}\mathord{\dots}\in\mathrm{W}_{\ast}$
and each $m\in\mathbb{N}_{+}$, let $[\omega]_{m}=\omega_{1}\omega_{2}\dots\omega_{m}$,
define $\mathbf{F}_{[\omega]_{m}}=\mathbf{F}_{\omega_{1}}\mathbf{F}_{\omega_{2}}\cdots\mathbf{F}_{\omega_{m}}$,
and $\mathbb{S}_{[\omega]_{m}}=\mathbf{F}_{[\omega]_{m}}\big(\mathbb{S}\big)$.
As a convention, we define $\mathbf{F}_{[\omega]_{0}}=\mathbf{Id}$.
The \emph{Hausdorff measure} on $\mathbb{S}$ is the unique Borel
probability measure $\nu$ on $\mathbb{S}$ such that $\nu\big(\mathbb{S}_{[\omega]_{m}}\big)=3^{-m}$
for all $\omega\in\mathrm{W}_{\ast},\,m\in\mathbb{N}$ , and the Hausdorff
measure on $\hat{\mathbb{S}}$ is the unique Borel measure $\hat{\nu}$
on $\hat{\mathbb{S}}$ such that $(\hat{\nu}\circ\tau_{i})|_{\mathbb{S}}=\nu$
for all $i\in\mathbb{Z}$.

\subsubsection*{Standard Dirichlet forms}

For each $m\in\mathbb{N}$ and any functions $u,v$ on $\bigcup_{m=0}^{\infty}\mathrm{V}_{m}$,
let
\begin{equation}
\mathcal{E}^{(m)}(u,v)=\sum_{\substack{x,y\in\mathrm{V}_{m}:|x-y|=2^{-m}}
}2^{-1}\,(5/3)^{m}\,[u(x)-u(y)][v(x)-v(y)].\label{eq:-63}
\end{equation}
The sequence $\{\mathcal{E}^{(m)}(u,u)\}_{m\in\mathbb{N}}$ is non-decreasing
(cf. \citep[Section 2.2 and Section 2.4]{Ki01}), therefore $\mathcal{E}(u,u)=\lim_{m\to\infty}\mathcal{E}^{(m)}(u,u)$
exists (possibly infinite), and the limit will be denoted by $\mathcal{E}(u)$
for simplicity.

Let $\mathcal{F}(\mathbb{S})=\big\{ u:u\ \text{is a function on}\ \bigcup_{m=0}^{\infty}\mathrm{V}_{m}\ \text{with}\ \mathcal{E}(u)<\infty\big\}$.
According to \citep[Theorem 2.2.6]{Ki01}, every function $u\in\mathcal{F}(\mathbb{S})$
uniquely extends to a continuous function on $\mathbb{S}$, in other
words, $\mathcal{F}(\mathbb{S})\subseteq C(\mathbb{S})$. $(\mathcal{E},\mathcal{F}(\mathbb{S}))$
is called the \emph{standard Dirichlet form }on $\mathbb{S}$, which
is a regular local Dirichlet form on $L^{2}(\mathbb{S};\nu)$. $\left(\mathcal{E},\mathcal{F}(\mathbb{S})\right)$
possesses the property of self-similarity in the sense that
\[
\mathcal{E}(u,v)=\sum_{i=1,2,3}(5/3)\,\mathcal{E}\big(u\circ\mathbf{F}_{i},v\circ\mathbf{F}_{i}\big),\;\;u,v\in\mathcal{F}(\mathbb{S}).
\]
Let $\mathcal{L}$ be the self-adjoint non-positive operator on $\mathrm{Dom}(\mathcal{L})\subseteq L^{2}(\mathbb{S};\nu)$
associated with $(\mathcal{E},\mathcal{F}(\mathbb{S}))$. 

Let $\mathcal{F}(\mathbb{S}\backslash\mathrm{V}_{0})=\big\{ u\in\mathcal{F}(\mathbb{S}):u\vert_{{\scriptscriptstyle \mathrm{V}_{0}}}=0\big\}$.
The restricted form $\big(\mathcal{E},\mathcal{F}(\mathbb{S}\backslash\mathrm{V}_{0})\big)$
is also a regular local Dirichlet form on $L^{2}(\mathbb{S};\nu)$
corresponding to Dirichlet boundary conditions.

By replacing $\mathrm{V}_{m}$ with $\hat{\mathrm{V}}_{m}$ in \eqref{eq:-63},
$\hat{\mathcal{E}}(u)$ can be defined similarly for any $u\in C(\hat{\mathbb{S}})$.
Let $\mathcal{F}(\hat{\mathbb{S}})$ be the completion of $\{u\in C(\hat{\mathbb{S}}):\hat{\mathcal{E}}(u)<\infty\}$
with respect to the norm $\hat{\mathcal{E}}(\cdot)^{1/2}+\Vert\cdot\Vert_{L^{2}(\hat{\nu})}$.
It can be shown that $\mathcal{F}(\hat{\mathbb{S}})\subseteq C_{0}(\hat{\mathbb{S}})$,
where $C_{0}(\hat{\mathbb{S}})$ is the space of continuous functions
on $\hat{\mathbb{S}}$ vanishing at infinity. $\big(\hat{\mathcal{E}},\mathcal{F}(\hat{\mathbb{S}})\big)$
is called the standard Dirichlet form on $\hat{\mathbb{S}}$, which
is a regular local Dirichlet form on $L^{2}(\hat{\mathbb{S}};\hat{\nu})$.
By definition
\begin{align}
\hat{\mathcal{E}}(u,v)=\sum_{i\in\mathbb{Z}}\mathcal{E}\big[(u\circ\tau_{i})\vert_{\mathbb{S}},(v\circ\tau_{i})\vert_{\mathbb{S}}\big],\;\;u,v\in\mathcal{F}(\hat{\mathbb{S}}).\label{eq:-21}
\end{align}
Similar to $\mathcal{E}$, the form $\hat{\mathcal{E}}$ is self-similar
in the sense that
\begin{equation}
\hat{\mathcal{E}}(u,v)=(5/3)\,\hat{\mathcal{E}}\big(u\circ\mathbf{F}_{1},v\circ\mathbf{F}_{1}\big),\;\;u,v\in\mathcal{F}(\hat{\mathbb{S}}).\label{eq:-62}
\end{equation}
For any $x,y\in\hat{\mathbb{S}}$, define $R(x,y)$ by 
\[
R(x,y)^{-1}=\inf\big\{\hat{\mathcal{E}}(u):u\in\mathcal{F}(\hat{\mathbb{S}}),\;u(x)=0,\;u(y)=1\big\}
\]
if $x\not=y$, and $R(x,y)=0$ if $x=y$. For every $x,y\in\hat{\mathbb{S}}$,
$R(x,y)<\infty$. Moreover, if $x\not=y$, then there exists a unique
$u\in\mathcal{F}(\hat{\mathbb{S}})$ such that $u(x)=1,\;u(y)=0,\;\hat{\mathcal{E}}(u)=R(x,y)^{-1}$
(see \citep[Theorem 2.3.4]{Ki01}). The function $R(\cdot,\cdot)$,
called the \emph{resistance metric}, is a metric on $\hat{\mathbb{S}}$
satisfying 
\[
C_{\ast}^{-1}\,|x-y|^{d_{w}-d_{f}}\le R(x,y)\le C_{\ast}\,|x-y|^{d_{w}-d_{f}},\;x,y\in\hat{\mathbb{S}}
\]
for some universal constant $C_{\ast}\ge1$, where $d_{s}=2\log3/\log5,\,d_{w}=\log5/\log2$,
and $d_{f}=d_{w}/(2d_{s})$ are the \emph{spectral dimension}, the
\emph{walk dimension}, and the \emph{fractal dimension} of $\hat{\mathbb{S}}$
respectively (cf. \citep[Lemma 3.3.5]{Ki01}). By the definition of
$R(\cdot,\cdot)$,
\begin{equation}
|u(x)-u(y)|\le R(x,y)^{1/2}\hat{\mathcal{E}}(u)^{1/2},\;\;u\in\mathcal{F}(\hat{\mathbb{S}}),\;x,y\in\hat{\mathbb{S}}.\label{eq:-55}
\end{equation}
Since $u\vert_{\mathbb{S}}\in\mathcal{F}(\mathbb{S}),\;u\in\mathcal{F}(\hat{\mathbb{S}})$
and $\max_{\mathbb{S}\times\mathbb{S}}R<\infty$, it follows from
\eqref{eq:-55} that
\begin{equation}
\mathop{\mathrm{osc}}_{\mathbb{S}}(u)\le C_{\ast}\,\mathcal{E}(u)^{1/2},\;\;u\in\mathcal{F}(\mathbb{S}).\label{eq:-14}
\end{equation}

Let $f$ be a function on $\mathrm{V}_{0}$. There exists a unique
$h\in\mathcal{F}(\mathbb{S})$ such that $h\vert_{{\scriptscriptstyle \mathrm{V}_{0}}}=f$
and $\mathcal{E}(h)=\inf\{\mathcal{E}(u):u\in\mathcal{F}(\mathbb{S}),\;u\vert_{{\scriptscriptstyle \mathrm{V}_{0}}}=f\}$
(see \citep[Proposition 3.2.1 ]{Ki01}). $h$ is called the \emph{harmonic
function} in $\mathbb{S}$ with boundary value $h\vert_{{\scriptscriptstyle \mathrm{V}_{0}}}=f$.
A harmonic function $h$, restricted on each $\mathrm{V}_{m},\;m\in\mathbb{N}$,
can be evaluated  by
\begin{equation}
(h\circ\mathbf{F}_{[\omega]_{m}})|_{{\scriptscriptstyle \mathrm{V}_{0}}}=\mathbf{A}_{[\omega]_{m}}(h\vert_{{\scriptscriptstyle \mathrm{V}_{0}}}),\;\;\omega=\omega_{1}\omega_{2}\omega_{3}\dots\,\in\mathrm{W}_{\ast},\label{eq:-30}
\end{equation}
(cf. \citep[Proposition 3.2.1 ]{Ki01}), where $\mathbf{A}_{[\omega]_{m}}=\mathbf{A}_{\omega_{m}}\cdots\mathbf{A}_{\omega_{2}}\mathbf{A}_{\omega_{1}}$,
with
\[
\mathbf{A}_{1}=\frac{1}{5}\left[\begin{array}{ccc}
5 & 0 & 0\\
2 & 2 & 1\\
2 & 1 & 2
\end{array}\right],\;\mathbf{A}_{2}=\frac{1}{5}\left[\begin{array}{ccc}
2 & 2 & 1\\
0 & 5 & 0\\
1 & 2 & 2
\end{array}\right],\;\mathbf{A}_{3}=\frac{1}{5}\left[\begin{array}{ccc}
2 & 1 & 2\\
1 & 2 & 2\\
0 & 0 & 5
\end{array}\right].
\]

Let $f$ be a function on $\mathrm{V}_{m}$. The \emph{$m$-harmonic
function} with boundary value $f$ is defined to be the unique $h\in\mathcal{F}(\mathbb{S})$
such that $h\vert_{{\scriptscriptstyle \mathrm{V}_{m}}}=f$ and that
$h\circ\mathbf{F}_{[\omega]_{m}}$ is a harmonic function for all
$\omega\in\mathrm{W}_{\ast}$. The energy of an $m$-harmonic function
$h$ can be calculated using $\mathcal{E}(h)=\mathcal{E}^{(m)}(h|_{{\scriptscriptstyle {\scriptscriptstyle \mathrm{V}_{m}}}}).$

\subsubsection*{Kusuoka measures and gradients}

Let $\mathbf{P}:\mathbb{R}^{3}\to\mathbb{R}^{3}$ be the projection
$\mathbf{P}x=x-(x_{1}+x_{2}+x_{3})/3$ for $x=(x_{1},x_{2},x_{3})\in\mathbb{R}^{3}$.
The \emph{Kusuoka measure $\mu$} on $\mathbb{S}$, as defined in
\citep{Ku89}, is the unique Borel probability measure on $\mathbb{S}$
such that
\[
\mu\big(\mathbb{S}_{[\omega]_{m}}\big)=2^{-1}\cdot\left(5/3\right)^{m}\,\mathrm{tr}\big(\mathbf{A}_{[\omega]_{m}}^{\mathrm{t}}\mathbf{P}\mathbf{A}_{[\omega]_{m}}\big)
\]
for all $\omega=\omega_{1}\omega_{2}\omega_{3}\dots\,\in\mathrm{W}_{\ast},\;m\in\mathbb{N}$.
The Kusuoka measure $\hat{\mu}$ on $\hat{\mathbb{S}}$ is the unique
Borel measure on $\hat{\mathbb{S}}$ such that $(\hat{\mu}\circ\tau_{i})|_{\mathbb{S}}=\mu$
for all $i\in\mathbb{Z}$.

The Kusuoka measure $\mu$ ($\hat{\mu}$ respectively) is singular
to the Hausdorff measure $\nu$ ($\hat{\nu}$ respectively) (cf. \citep[p. 678]{Ku89}).
If $u\in\mathcal{F}(\mathbb{S})$, then $\mu_{\langle u\rangle}$
denotes the \emph{energy measure} of $u$, i.e. the Borel measure
on $\mathbb{S}$ such that $\int_{\mathbb{S}}\phi\,d\mu_{\langle u\rangle}=\mathcal{E}(\phi u,u)-2^{-1}\mathcal{E}(\phi,u^{2})$
for $\phi\in\mathcal{F}(\mathbb{S})$. By \citep[Theorem (5.4)]{Ku89},
$\mu_{\langle u\rangle}\ll\mu$ for all $u\in\mathcal{F}(\mathbb{S})$
(see \citep{Hin08,Hin13} for similar results on general fractals).
Moreover, there exists a unique linear operator $\nabla:\mathcal{F}(\mathbb{S})\to L^{2}(\mu)$,
called the \emph{gradient} \emph{operator} on $\mathbb{S}$, satisfying
the following: (i) $\mu_{\langle u\rangle}=|\nabla u|^{2}\,\mu$ for
all $u\in\mathcal{F}(\mathbb{S})$, and (ii) if $h$ is the harmonic
function with boundary value  $h(0,0)=0,\,h(1,0)=h(1/2,\sqrt{3}/2)=1$,
then $\nabla h>0\;\;\mu$-a.e. The gradient operator $\nabla:\mathcal{F}(\hat{\mathbb{S}})\to L^{2}(\hat{\mu})$
on \emph{$\hat{\mathbb{S}}$} is defined by $[(\nabla u)\circ\tau_{i}]\vert_{\mathbb{S}}=\nabla[(u\circ\tau_{i})\vert_{\mathbb{S}}]$
for all $u\in\mathcal{F}(\hat{\mathbb{S}})$ and all $i\in\mathbb{Z}$.
\begin{rem}
\label{rem:-9}We should point out that there exist several slight
variants of gradients on fractals, which are introduced to address
different problems (see, e. g. \citep{Ki93,Tep00,Str00,CS03,Hin10,BK16,LQ16}
and references therein). The definition of gradients on $\mathbb{S}$
adopted in the present paper was introduced in \citep{LQ16} via martingale
representations, and can be regarded as the special case of the definition
given in \citep{Hin10}, where $\mu$ is the minimal energy-dominant
measure (see \citep[p. 3]{Hin10} for the definition).
\end{rem}

\section{\label{sec:}Sobolev inequalities}

The objective of this section is to establish some Sobolev inequalities
involving different (probably mutually singular) measures on $\mathbb{S}$
and $\hat{\mathbb{S}}$ (Theorem \ref{thm:-2} and Theorem \ref{thm:-4}
respectively), which is crucial to our study of some semi-linear parabolic
equations on the gasket. A sufficient and necessary condition for
the validity of these Sobolev inequalities (Theorem \ref{thm:-1}
and Theorem \ref{thm:-3}) will be established as well.

To shed some light on the motivation of these inequalities, consider
the following simple parabolic PDE on $\mathbb{S}$
\[
\partial_{t}u\,d\nu=\mathcal{L}u\,d\nu+\nabla u\,d\mu.
\]
Here the singular measures $\nu$ and $\mu$ must be involved as $\mathcal{L}u$
is $\nu$-a.e. defined while $\nabla u$ is only $\mu$-a.e. defined.
A precise interpretation of this equation will be given in Section
\ref{sec:-1}. Let us assume for the moment that if $u$ is a solution
then one may test the equation against the solution to obtain
\[
\frac{d}{dt}\Vert u(t)\Vert_{L^{2}(\nu)}^{2}=-\mathcal{E}(u(t))+\langle\nabla u(t),u(t)\rangle_{\mu},
\]
from which it follows that
\[
\frac{d}{dt}\Vert u(t)\Vert_{L^{2}(\nu)}^{2}\le-\frac{1}{2}\mathcal{E}(u(t))+\frac{1}{2}\Vert u(t)\Vert_{L^{2}(\mu)}^{2}.
\]
For PDEs on Euclidean spaces, the measures $\nu$ and $\mu$ are equal
to the Lebesgue measures, and therefore, the above differential inequality
together with Grönwall's inequality yields the energy estimates and
the existence and uniqueness of solutions. However, on $\mathbb{S}$,
the measures $\nu$ and $\mu$ are mutually singular, and hence the
$L^{2}$-norms $\Vert\cdot\Vert_{L^{2}(\nu)}$ and $\Vert\cdot\Vert_{L^{2}(\mu)}$
are in general incomparable. Thus, Grönwall's inequality does not
apply in this case. For PDEs involving gradients on $\mathbb{S}$,
an appropriate comparison of $\Vert\cdot\Vert_{L^{2}(\nu)}$ and $\Vert\cdot\Vert_{L^{2}(\mu)}$
is necessary to obtaining energy estimates. In fact, for functions
$u\in\mathcal{F}(\mathbb{S})$, the $L^{2}$-norms $\Vert u\Vert_{L^{2}(\nu)}$
and $\Vert u\Vert_{L^{2}(\mu)}$ must be compared with the involvement
of (an arbitrarily small portion of) the energy $\mathcal{E}(u)$
(see Corollary \ref{cor:} below). This type of comparison is possible
due to the Sobolev inequalities to be established in this section.

For convenience, $C_{\ast}$ will always denote a generic universal
constant which may be different on various occasions.
\begin{defn}
Let $S_{i,m}=2^{m}\tau_{i}\big(\mathbb{S}\big),\;m,i\in\mathbb{Z}$.
The \emph{energy of $u\in\mathcal{F}(\hat{\mathbb{S}})$ on $S_{i,m}$
}is defined to be\emph{ $\hat{\mathcal{E}}|_{S_{i,m}}(u)=(3/5)^{m}\,\mathcal{E}[(u\circ\tau_{i}\circ\mathbf{F}_{1}^{-m})|_{\mathbb{S}}].$}
\end{defn}
Clearly, $\hat{\mathbb{S}}$ can be written as the non-overlapping
union $\hat{\mathbb{S}}=\bigcup_{i\in\mathbb{Z}}S_{i,m}$ for each
$m\in\mathbb{Z}$. Therefore, $\hat{\mathcal{E}}(u)=\sum_{i\in\mathbb{Z}}\hat{\mathcal{E}}|_{S_{i,m}}(u)$
for any $u\in\mathcal{F}(\hat{\mathbb{S}})$ in view of \eqref{eq:-21}
and \eqref{eq:-62}.
\begin{defn}
The constant $\delta_{s}>0$ is defined by $1/\delta_{s}=2/d_{s}-1=\log5/\log3-1$.
\end{defn}
The constant $\delta_{s}$ is defined so that $5/3=3^{1/\delta_{s}}$.
Therefore, for every $i$ and $m$, by \eqref{eq:-14},
\[
\begin{aligned}\mathop{\mathrm{osc}}_{\mathbb{S}}\big( & u\circ\tau_{i}\circ\mathbf{F}_{1}^{-m}\big)\le C_{\ast}\,\mathcal{E}[(u\circ\tau_{i}\circ\mathbf{F}_{1}^{-m})|_{\mathbb{S}}]^{1/2}\\
 & =C_{\ast}\,(5/3)^{m/2}\,\hat{\mathcal{E}}|_{S_{i,m}}(u)^{1/2}=C_{\ast}\,\hat{\nu}\big(S_{i,m}\big)^{1/(2\delta_{s})}\hat{\mathcal{E}}|_{S_{i,m}}(u)^{1/2},
\end{aligned}
\]
which implies that
\begin{equation}
\mathop{\mathrm{osc}}_{S_{i,m}}(u)\le C\,\hat{\nu}\big(S_{i,m}\big)^{1/(2\delta_{s})}\hat{\mathcal{E}}|_{S_{i,m}}(u)^{1/2}.\label{eq:-22}
\end{equation}

\begin{defn}
A subset $S\subseteq\hat{\mathbb{S}}$ is called a \emph{dyadic triangle}
if $S=S_{i,m}$ for some $m,i\in\mathbb{Z}$.
\end{defn}
We are now in a position to formulate the main results of this section.
Let $\hat{\sigma}$ be a Borel measure on $\hat{\mathbb{S}}$ satisfying
the following condition: there exist constants $C_{\hat{\sigma}}\ge1$
and $0<\underbar{\ensuremath{\delta}}\le\bar{\delta}\le\infty,\,\bar{\delta}\ge1$
such that
\begin{equation}
\left\{ \begin{aligned}\hat{\sigma}(S)\le C_{\hat{\sigma}}\,\hat{\nu}(S)^{1/\bar{\delta}}, & \;\;\text{if}\;\ 0<\text{diam}(S)<1,\\
\hat{\sigma}(S)\le C_{\hat{\sigma}}\,\hat{\nu}(S)^{1/\underbar{\ensuremath{{\scriptstyle \delta}}}}, & \;\;\text{if}\;\;\text{diam}(S)\ge1,
\end{aligned}
\right.\tag{{\text{M.1}}}\label{eq:-47}
\end{equation}
for any dyadic triangle $S\subseteq\hat{\mathbb{S}}$, where $\text{diam}(A)$
denotes the diameter of $A\subseteq\hat{\mathbb{S}}$ with respect
to the Euclidean metric.
\begin{rem}
\label{rem:}(i) In literature, a Borel measure $\sigma$ on $\mathbb{R}^{n}$
is called an Ahlfors regular measure if there exists a $d>0$ such
that $C^{-1}r^{d}\le\sigma(B(x,r))\le Cr^{d}$ for any ball of radius
$r$ centred at $x\in\mathbb{R}^{n}$. Therefore, the Hausdorff measure
is an Ahlfors regular measure, and we think it is appropriate to call
the measure $\hat{\sigma}$ in \eqref{eq:-47} an Ahlfors upper regular
measure (with distinct exponents for expansion and contraction). In
the present paper, we formulate the condition \eqref{eq:-47} in terms
of the Hausdorff measure rather than diameter of sets because of notational
convenience when comparing measures. We would like to mention that
a heat kernel estimate implies the Ahlfors regularity of the Hausdorff
measure (see e.g. \citep[Theorem 3.2]{GHL03} and references therein).

\ (ii) The restriction $\bar{\delta}\ge1$ in the condition \eqref{eq:-47}
is necessary in view of the countable additivity of measures.

(iii) Notice that we do not require \eqref{eq:-47} to hold for general
Borel sets. In fact, \eqref{eq:-47} being valid for all Borel sets
implies the absolute continuity of $\hat{\sigma}$ with respect to
$\hat{\nu}$.
\end{rem}
We would like to point out that the condition \eqref{eq:-47} is general
enough to include many cases of interests, some of important examples
are listed below.
\begin{example}
\label{exmp:}(i) Dirac measures, for which the condition \eqref{eq:-47}
is satisfied with $\underbar{\ensuremath{\delta}}=\bar{\delta}=\infty$.

\ (ii) The Kusuoka measure $\hat{\mu}$ (cf. Corollary \ref{cor:-1}.(b)
below).

(iii) Analogues on $\hat{\mathbb{S}}$ of weighted measures $|x|^{-\theta}\,dx$
on $\mathbb{R}^{d}$ with $0\le\theta<d$. For any cube $Q\subseteq\mathbb{R}^{d}$,
we have $\int_{Q}|x|^{-\theta}dx\le C\,|Q|^{1-\theta/d}$ for some
constant $C>0$ depending only on $d$. Therefore, the analogue on
$\hat{\mathbb{S}}$ of $|x|^{-\theta}\,dx$ on $\mathbb{R}^{d}$ would
be a Borel measure $\hat{\sigma}\ll\hat{\nu}$ satisfying the condition
\eqref{eq:-47} with $\underbar{\ensuremath{\delta}},\bar{\delta}$
given by $1/\underbar{\ensuremath{\delta}}=1/\bar{\delta}=1-\theta/d_{s}$.
Here we have used $d_{s}$ as the Sobolev dimension of $\hat{\mathbb{S}}$
(cf. Remark \ref{rem:-1} below).
\end{example}
\begin{thm}
\label{thm:-2}Let $1\le p\le q\le\infty,\;q\ge2$. Suppose $\hat{\sigma}$
is a Borel measure on $\hat{\mathbb{S}}$ satisfying the condition
\eqref{eq:-47}. Then
\begin{equation}
\Vert u\Vert_{L^{q}(\hat{\sigma})}\le C\,\sum_{i=1,2}\hat{\mathcal{E}}(u)^{a_{i}/2}\Vert u\Vert_{L^{p}(\hat{\nu})}^{1-a_{i}},\;\;u\in\mathcal{F}(\hat{\mathbb{S}}),\label{eq:-43}
\end{equation}
where 
\begin{equation}
a_{1}=\Big[\frac{1/p-1/(q\bar{\delta})}{1/p+1/(2\delta_{s})}\Big]^{+},\;a_{2}=\Big[\frac{1/p-1/(q\underbar{\ensuremath{\delta}})}{1/p+1/(2\delta_{s})}\Big]^{+},\label{eq:-6}
\end{equation}
and $C>0$ is a constant depending only on the constant $C_{\hat{\sigma}}$
in \eqref{eq:-47}.

Moreover, if there exists a sequence $\{S_{m}\}_{m\in\mathbb{Z}}$
of dyadic triangles such that 
\[
\lim_{m\to-\infty}\mathrm{diam}(S_{m})=0\textrm{, }\;\lim_{m\to\infty}\mathrm{diam}(S_{m})=\infty
\]
and
\begin{equation}
1/\bar{\delta}=\lim_{m\to-\infty}\frac{\log\hat{\sigma}(S_{m})}{\log\hat{\nu}(S_{m})},\;1/\underbar{\ensuremath{\delta}}=\lim_{m\to\infty}\frac{\log\hat{\sigma}(S_{m})}{\log\hat{\nu}(S_{m})},\tag{{\text{M.2}}}\label{eq:-38}
\end{equation}
then the pair of exponents given by \eqref{eq:-6} is optimal in the
following sense: if
\begin{equation}
\Vert u\Vert_{L^{q}(\hat{\sigma})}\le C\,\sum_{i=1}^{N}\hat{\mathcal{E}}(u)^{b_{i}/2}\Vert u\Vert_{L^{p}(\hat{\nu})}^{1-b_{i}},\;\;u\in\mathcal{F}(\hat{\mathbb{S}}),\label{eq:-10}
\end{equation}
for some constants $b_{i}\in[0,1],\;1\le i\le N,\;N\in\mathbb{N}_{+}$
and $C>0$ independent of $u$, then $\min_{i}b_{i}\le a_{2}\le a_{1}\le\max_{i}b_{i}$.
\end{thm}
\begin{proof}
Suppose first that $p\le q<\infty$. Let $\hat{\nu}_{m}=\hat{\nu}(2^{m}\mathbb{S})=3^{m}$,
$S_{i,m}=2^{m}\,\tau_{i}(\mathbb{S})$ for any $m,\,i\in\mathbb{Z}$.
Then $\hat{\mathbb{S}}=\bigcup_{i}S_{i,m}$. When $m\ge0$, by \eqref{eq:-22},
we have that
\[
\begin{aligned}\int_{\hat{\mathbb{S}}}|u|^{q}d\hat{\sigma} & \le2^{q-1}\sum_{i}\Big[\int_{S_{i,m}}\Big|u-\frac{1}{\hat{\nu}_{m}}\int_{S_{i,m}}u\,d\hat{\nu}\Big|^{q}\,d\hat{\sigma}+\Big|\frac{1}{\hat{\nu}_{m}}\int_{S_{i,m}}u\,d\hat{\nu}\Big|^{q}\,\hat{\sigma}(S_{i,m})\Big]\\
 & \le C^{q}\sum_{i}\Big[\hat{\nu}_{m}^{q/(2\delta_{s})}\hat{\mathcal{E}}|_{S_{i,m}}(u)^{q/2}\,\hat{\sigma}(S_{i,m})+\frac{1}{\hat{\nu}_{m}^{q/p}}\Big[\int_{S_{i,m}}|u|^{p}\,d\hat{\nu}\Big]^{q/p}\,\hat{\sigma}(S_{i,m})\Big]\\
 & \le C^{q}\sum_{i}\Big[\hat{\nu}_{m}^{q/(2\delta_{s})+1/\underbar{\ensuremath{{\scriptstyle \delta}}}}\,\hat{\mathcal{E}}|_{S_{i,m}}(u)^{q/2}+\hat{\nu}_{m}^{1/\underbar{\ensuremath{{\scriptstyle \delta}}}-q/p}\,\Big(\int_{S_{i,m}}|u|^{p}\,d\hat{\nu}\Big)^{q/p}\Big]\\
 & \le C^{q}\Big\{\hat{\nu}_{m}^{q/(2\delta_{s})+1/\underbar{\ensuremath{{\scriptstyle \delta}}}}\,\Big[\sum_{i}\hat{\mathcal{E}}|_{S_{i,m}}(u)\Big]^{q/2}+C^{q}\,\hat{\nu}_{m}^{1/\underbar{\ensuremath{{\scriptstyle \delta}}}-q/p}\,\Big[\sum_{i}\int_{S_{i,m}}|u|^{p}\,d\hat{\nu}\Big]^{q/p}\Big\}\\
 & =C^{q}\,\big[\hat{\nu}_{m}^{q/(2\delta_{s})+1/\underbar{\ensuremath{{\scriptstyle \delta}}}}\,\hat{\mathcal{E}}(u)^{q/2}+\hat{\nu}_{m}^{1/\underbar{\ensuremath{{\scriptstyle \delta}}}-q/p}\,\Vert u\Vert_{L^{p}(\hat{\nu})}^{q}\big],
\end{aligned}
\]
where and hereafter $C>0$ denotes a generic constant depending only
on the constant $C_{\hat{\sigma}}$ in \eqref{eq:-47}. Therefore,
\begin{equation}
\Vert u\Vert_{L^{q}(\hat{\sigma})}\le C\,\big[\hat{\nu}_{m}^{1/(2\delta_{s})+1/(q\underbar{\ensuremath{{\scriptstyle \delta}}})}\,\hat{\mathcal{E}}(u)^{1/2}+\hat{\nu}_{m}^{1/(q\underbar{\ensuremath{{\scriptstyle \delta}}})-1/p}\,\Vert u\Vert_{L^{p}(\hat{\nu})}\big].\label{eq:-71}
\end{equation}
Similarly, when $m\le0$, we have that
\begin{equation}
\Vert u\Vert_{L^{q}(\hat{\sigma})}\le C\,\big[\hat{\nu}_{m}^{1/(2\delta_{s})+1/(q\bar{\delta})}\,\hat{\mathcal{E}}(u)^{1/2}+\hat{\nu}_{m}^{1/(q\bar{\delta})-1/p}\,\Vert u\Vert_{L^{p}(\hat{\nu})}\big].\label{eq:-66}
\end{equation}

\medskip{}

The proof of \eqref{eq:-43} is done by optimising over $m$. Suppose
$\hat{\mathcal{E}}(u)^{1/2}\ge\Vert u\Vert_{L^{p}(\hat{\nu})}$. Consider
the following two cases:

\medskip{}

\noindent \emph{Case 1:} $p\le q\le p/\bar{\delta}$.\enskip{}Note
that $p\le p/\bar{\delta}$ forces $\bar{\delta}=1$ and therefore
$1/(q\bar{\delta})=1/p,\;a_{1}=0$. Setting $m\to-\infty$ in \eqref{eq:-66}
gives that
\[
\Vert u\Vert_{L^{q}(\hat{\sigma})}\le C\,\hat{\mathcal{E}}(u)^{a_{1}/2}\Vert u\Vert_{L^{p}(\hat{\nu})}^{1-a_{1}}.
\]

\noindent \emph{Case 2:} $q>p/\bar{\delta}$.\enskip{}Setting
\[
m=\sup\big\{ m\le0:\hat{\nu}_{m}^{1/(2\delta_{s})+1/p}\le\hat{\mathcal{E}}(u)^{-1/2}\Vert u\Vert_{L^{p}(\hat{\nu})}\le1\big\}<\infty,
\]
in \eqref{eq:-66}, we obtain that
\[
\Vert u\Vert_{L^{q}(\hat{\sigma})}\le C\,\hat{\mathcal{E}}(u)^{a_{1}/2}\Vert u\Vert_{L^{p}(\hat{\nu})}^{1-a_{1}},\;\text{where}\ a_{1}=\Big[\frac{1/p-1/(q\bar{\delta})}{1/p+1/(2\delta_{s})}\Big]^{+}.
\]

Suppose that $\hat{\mathcal{E}}(u)^{1/2}<\Vert u\Vert_{L^{p}(\hat{\nu})}$.
We consider the two cases.

\medskip{}

\noindent \emph{Case 1:} $p\le q\le p/\underbar{\ensuremath{\delta}}$.\enskip{}In
this case, $a_{2}=0$. Setting $m=0$ in \eqref{eq:-71} gives that
\[
\Vert u\Vert_{L^{q}(\hat{\mu})}\le C\,\hat{\mathcal{E}}(u)^{a_{2}/2}\Vert u\Vert_{L^{p}(\hat{\nu})}^{1-a_{2}}.
\]

\noindent \emph{Case 2:} $q>p/\underbar{\ensuremath{\delta}}$.\enskip{}Setting
\[
m=\inf\big\{ m\ge0:\hat{\nu}_{m}^{1/(2\delta_{s})+1/p}\ge\hat{\mathcal{E}}(u)^{-1/2}\Vert u\Vert_{L^{p}(\hat{\nu})}>1\big\}<\infty,
\]
in \eqref{eq:-71}, we obtain that
\[
\Vert u\Vert_{L^{q}(\hat{\sigma})}\le C\,\hat{\mathcal{E}}(u)^{a_{2}/2}\Vert u\Vert_{L^{p}(\hat{\nu})}^{1-a_{2}},\;\text{where}\ a_{2}=\Big[\frac{1/p-1/(q\underbar{\ensuremath{\delta}})}{1/p+1/(2\delta_{s})}\Big]^{+}.
\]
This proves \eqref{eq:-43} for $q<\infty$. Setting $q\to\infty$
proves the case when $q=\infty$ as the constant $C$ is independent
of $q$.

\medskip{}

Suppose in addition that the condition \eqref{eq:-38} is satisfied,
we prove that $(a_{1},a_{2})$ is the optimal pair of exponents. We
first show that, for any dyadic triangle $S\subseteq\hat{\mathbb{S}}$,
there exists an $h_{S}\in\mathcal{F}(\hat{\mathbb{S}})$ such that
\begin{equation}
C_{\ast}^{-1}\le h_{S}\le C_{\ast}\ \text{on}\ S,\;\text{supp}(h_{S})\subseteq\tilde{S},\;\text{and}\ \hat{\mathcal{E}}(h_{S})\le C_{\ast}\,\hat{\nu}(S)^{-1/\delta_{s}},\label{eq:-36}
\end{equation}
where $\tilde{S}=\{x\in\hat{\mathbb{S}}:\text{dist}(x,S)\le\text{diam}(S)\}$.

To see this, suppose first that $S=2^{-1}\mathbb{S}$ for some $m\in\mathbb{Z}$.
Let $h$ be the $1$-harmonic function in $\mathbb{S}$ with boundary
value
\[
h\big|_{\mathrm{V}_{1}}(x)=\left\{ \begin{aligned}1, & \;\;\text{if}\;x=(0,0),\\
0, & \;\;\text{otherwise}.
\end{aligned}
\right.
\]
Let $h(x)=h(-x)$ for $x\in-\mathbb{S}$, and $h(x)=0$ for $x\in\hat{\mathbb{S}}\backslash[\mathbb{S}\cup(-\mathbb{S})]$.
Then $h\in\mathcal{F}(\hat{\mathbb{S}})$ and satisfies \eqref{eq:-36}.
For a general dyadic triangle $S=2^{m}\tau_{i}(\mathbb{S}),\;i,m\in\mathbb{Z}$,
let $h_{S}=h\circ\tau_{i}^{-1}\circ\mathbf{F}_{1}^{m}$. Then $h_{S}\in\mathcal{F}(\hat{\mathbb{S}})$,
and the property \eqref{eq:-36} follows from \eqref{eq:-21} and
the self-similar property \eqref{eq:-62}.

Suppose that \eqref{eq:-10} holds. Let $\{S_{m}\}_{m\in\mathbb{Z}}$
be the sequence of dyadic triangles in \eqref{eq:-38}. For each $m\in\mathbb{Z}$,
by the above, there exists an $h_{m}\in\mathcal{F}(\hat{\mathbb{S}})$
such that $h_{m}\sim1$ on $S_{m}$, $\text{supp}(h_{m})\subseteq\tilde{S}_{m}$
and $\hat{\mathcal{E}}(h_{m})\lesssim\hat{\nu}(S_{m})^{-1/\delta_{s}}$,
where the notation $A\lesssim B$ means that $A\le cB$ for some constant
$c>0$ independent of $m$, and $A\simeq B$ means that $A\lesssim B$
and $B\lesssim A$. In view of \eqref{eq:-38}, it is easily seen
that
\[
\Vert h_{m}\Vert_{L^{q}(\hat{\sigma})}\simeq\hat{\nu}(S_{m})^{1/(q\underbar{\ensuremath{{\scriptstyle \delta}}})},\;\Vert h_{m}\Vert_{L^{p}(\hat{\nu})}\simeq\hat{\nu}(S_{m})^{1/p},\;m\to\infty,
\]
and
\[
\Vert h_{m}\Vert_{L^{q}(\hat{\sigma})}\simeq\hat{\nu}(S_{m})^{1/(q\bar{\delta})},\;\Vert h_{m}\Vert_{L^{p}(\hat{\nu})}\simeq\hat{\nu}(S_{m})^{1/p},\;m\to-\infty.
\]
It follows from the above and \eqref{eq:-10} that
\[
\hat{\nu}(S_{m})^{1/(q\underbar{\ensuremath{{\scriptstyle \delta}}})}\lesssim\sum_{i=1}^{N}\hat{\nu}(S_{m})^{-b_{i}/(2\delta_{s})+(1-b_{i})/p},\;\;m\to\infty,
\]
and
\[
\hat{\nu}(S_{m})^{1/(q\bar{\delta})}\lesssim\sum_{i=1}^{N}\hat{\nu}(S_{m})^{-b_{i}/(2\delta_{s})+(1-b_{i})/p},\;\;m\to-\infty,
\]
These inequalities imply that $\min_{i}b_{i}\le a_{2}\le a_{1}\le\max_{i}b_{i}$.
\end{proof}
\begin{rem}
\label{rem:-1}(i) Some comments are desired on the interpretation
of the exponents appearing in the inequality \eqref{eq:-43}. Recall
that, on Euclidean space $\mathbb{R}^{d}$, the celebrated Gagliardo\textendash Nirenberg
inequality takes the form
\[
\Vert D^{j}u\Vert_{L^{q}(\mathbb{R}^{d})}\le C\Vert D^{m}u\Vert_{L^{r}(\mathbb{R}^{d})}^{a}\Vert u\Vert_{L^{p}(\mathbb{R}^{d})}^{1-a}
\]
where $a\in[0,1]$ is given by $\frac{1}{q}=\frac{j}{d}+\big(\frac{1}{r}-\frac{m}{d}\big)a+\frac{1-a}{p}$.
The case corresponding to setting of Dirichlet forms is the one when
$j=0,\;m=1$ and $r=2$, for which the exponent $a$ is given by
\begin{equation}
a=\frac{1/p-1/q}{1/p-1/2+1/d}.\label{eq:-29}
\end{equation}
Some insights are gained by comparing \eqref{eq:-6} and \eqref{eq:-29}:

a) The exponents $a_{i},\,i=1,2$ in \eqref{eq:-43} are determined
by the harmonic structure on $\hat{\mathbb{S}}$ (or equivalently
the Dirichlet form $\hat{\mathcal{E}}$), the configuration parameters
$\underbar{\ensuremath{\delta}}$ and $\bar{\delta}$ of the measure
$\hat{\sigma}$, and the embedding parameters $p$ and $q$.

b) The effective Sobolev dimension $d$ of $\hat{\mathbb{S}}$, if
exists, should depend only on the harmonic structure. This dependence
is expressed in \eqref{eq:-6} as the denominator $1/p+1/(2\delta_{s})$.
Comparing this to the denominator of \eqref{eq:-29}, we see that
the Sobolev dimension $d$ should be given by $1/p-1/2+1/d=1/p+1/(2\delta_{s})$,
i.e. $d=d_{s}$. This suggests the identification of the spectral
dimension $d_{s}$ as the effective Sobolev dimension of $\hat{\mathbb{S}}$.
See \citep[pp. 44--45]{Str03} for more comments on $d_{s}$.

\ (ii) The inequality \eqref{eq:-10} includes the analogue on $\hat{\mathbb{S}}$
of a specific case of the weighted Sobolev inequalities on $\mathbb{R}^{d}$
in \citep{CKN84}. The weighted Sobolev inequalities established in
\citep{CKN84} take the form 
\[
\Vert|x|^{\gamma}u\Vert_{L^{q}(\mathbb{R}^{d})}\le C\Vert|x|^{\alpha}Du\Vert_{L^{r}(\mathbb{R}^{d})}^{a}\Vert|x|^{\beta}u\Vert_{L^{p}(\mathbb{R}^{d})}^{1-a}
\]
where $\alpha,\beta,\gamma<0$ satisfy $1/r+\alpha/d>0$, $1/p+\beta/d\ge1/q+\gamma/d>0$
and $\frac{1}{q}+\frac{\gamma}{d}=a\big(\frac{1}{r}+\frac{\alpha-1}{d}\big)+(1-a)\big(\frac{1}{p}+\frac{\beta}{d}\big)$.
The case corresponding to setting of Dirichlet forms is the one when
$\alpha=\beta=0,\,r=2$ and $1/p\ge1/q+\gamma/d_{s}>0$, for which
the weighted inequality reads
\begin{equation}
\Vert u\Vert_{L^{q}(|x|^{\gamma q}dx)}\le C\,\Vert Du\Vert_{L^{2}(dx)}^{a}\,\Vert u\Vert_{L^{p}(dx)}^{1-a}.\label{eq:-45}
\end{equation}
As remarked in Example \ref{exmp:}.(iii), the analogue on $\hat{\mathbb{S}}$
of $|x|^{\gamma q}\,dx$ on $\mathbb{R}^{d}$ is a Borel measure $\hat{\sigma}$
on $\hat{\mathbb{S}}$ satisfying the condition \eqref{eq:-47} with
$\underbar{\ensuremath{\delta}},\bar{\delta}$ given by $1/\underbar{\ensuremath{\delta}}=1/\bar{\delta}=1+\gamma q/d_{s}$.
Therefore, the analogue of \eqref{eq:-45} on $\hat{\mathbb{S}}$
should be $\Vert u\Vert_{L^{q}(\hat{\sigma})}\le C\,\hat{\mathcal{E}}(u)^{a/2}\Vert u\Vert_{L^{p}(\hat{\nu})}^{1-a}$
with $a$ given by $\frac{1}{q}+\frac{\gamma}{d_{s}}=a\big(\frac{1}{2}-\frac{1}{d_{s}}\big)+\frac{1-a}{p}$.
This coincides with the result of \eqref{eq:-26} since the exponents
for the measure $\hat{\sigma}$ are given by $a_{1}=a_{2}=\frac{1/p-1/q-\gamma/d_{s}}{1/p+1/d_{s}-1/2}=a$.

(iii) An additive version of \eqref{eq:-43}, which is a corollary
of \eqref{eq:-43} and Young's inequality, is derived in \citep{HR15}
for the study of vector fields on resistance spaces.
\end{rem}
\medskip{}

According to Theorem \ref{thm:-2}, the condition \eqref{eq:-47}
is sufficient for the derivation of Sobolev inequalities. The following
theorem states that this condition is also necessary for the validity
of Sobolev inequalities of the form \eqref{eq:-10} with $q<\infty$.
\begin{thm}
\label{thm:-1}Let $\hat{\sigma}$ be a Borel measure on $\hat{\mathbb{S}}$.
Suppose that there exist some constants $p,q\in(0,\infty)$, $b_{i}\in[0,1]$
where $1\le i\le N$, and $C>0$, such that \eqref{eq:-10} holds
for all $u\in\mathcal{F}(\hat{\mathbb{S}})$. Then there exist constants
$0<\underbar{\ensuremath{\delta}}\le\bar{\delta}\le\infty$ such that
the condition \eqref{eq:-47} is satisfied.
\end{thm}
\begin{proof}
Suppose that \eqref{eq:-10} holds. For any dyadic triangle $S\subseteq\hat{\mathbb{S}}$,
as shown in the proof of Theorem \ref{thm:-2}, there exists a piecewise
harmonic function $h_{S}\in\mathcal{F}(\hat{\mathbb{S}})$ such that
\[
h_{S}\simeq1\ \text{on}\ S,\;\text{supp}(h_{S})\subseteq\tilde{S},\;\text{and}\ \hat{\mathcal{E}}(h_{S})\lesssim\hat{\nu}(S)^{-1/\delta_{s}},
\]
where the notation $\tilde{S}$ and the relations $\lesssim$ and
$\simeq$ are the same as those in the proof of Theorem \ref{thm:-2}.
Applying \eqref{eq:-10} to $h_{S}$ gives that
\begin{equation}
\hat{\sigma}(S)^{1/q}\lesssim\sum_{i}\hat{\nu}(S)^{-b_{i}/(2\delta_{s})+(1-b_{i})/p},\label{eq:-48}
\end{equation}
Since $q<\infty$, it follows from the above that
\[
\sup\big\{\hat{\sigma}(S):S\ \text{is a dyadic triangle with}\ \text{diam}(S)=1\big\}<\infty.
\]
Therefore, the first part of \eqref{eq:-47} is satisfied with $\bar{\delta}=\infty$.

Furthermore, for any dyadic triangle $S$ with $\text{diam}(S)\ge1$,
by \eqref{eq:-48}, $\hat{\sigma}(S)^{1/q}\lesssim\hat{\nu}(S)^{1/p}$
as $\hat{\nu}(S)\ge1$. Setting $\underbar{\ensuremath{\delta}}=p/q$
completes the proof.
\end{proof}
Applying Theorem \ref{thm:-2} to the cases when $\hat{\sigma}=\hat{\nu}$
and when $\hat{\sigma}=\hat{\mu}$, we obtain the following.
\begin{cor}
\label{cor:-1}Let $1\le p\le q\le\infty,\;q\ge2$. Then\\
\\
(a) The inequality \eqref{eq:-43} holds with $\hat{\sigma}=\hat{\nu}$
and $a_{1}=a_{2}=\frac{1/p-1/q}{1/p+1/(2\delta_{s})}\in[0,1)$. In
particular,
\begin{equation}
\max_{\hat{\mathbb{S}}}u\le C\,\hat{\mathcal{E}}(u)^{a/2}\Vert u\Vert_{L^{p}(\hat{\nu})}^{1-a},\;\;u\in\mathcal{F}(\hat{\mathbb{S}}),\label{eq:-44}
\end{equation}
with $a=\frac{1/p}{1/p+1/(2\delta_{s})}$. Conversely, the inequality
\eqref{eq:-43} holds for all $u\in\mathcal{F}(\hat{\mathbb{S}})$
if and only if $a_{1}=a_{2}=\frac{1/p-1/q}{1/p+1/(2\delta_{s})}$.\\
\\
(b) The inequality \eqref{eq:-43} holds with $\hat{\sigma}=\hat{\mu}$.
The pair $(a_{1},a_{2})$ given by \eqref{eq:-6} is optimal, where
$\underbar{\ensuremath{\delta}}=1$ and $\bar{\delta}=\delta_{s}$.
\end{cor}
\begin{proof}
The only thing needs a proof is that $\bar{\delta}=\delta_{s}$ in
(b). Clearly,
\[
1/\bar{\delta}=\inf_{\omega\in\mathrm{W}_{\ast}}\liminf_{m\to\infty}\Big[-\frac{1}{m\,\log3}\log\mu\big(\mathbb{S}_{[\omega]_{m}}\big)\Big].
\]
We show that
\begin{equation}
\sup_{\omega\in\mathrm{W}_{\ast}}\lim_{m\to\infty}\big[\mathrm{tr}\big(\mathbf{A}_{[\omega]_{m}}^{\mathrm{t}}\mathbf{P}\mathbf{A}_{[\omega]_{m}}\big)\big]^{1/m}=(3/5)^{2},\label{eq:-23}
\end{equation}
from which the conclusion follows immediately.

Let $\mathbf{Y}_{i}=\mathbf{P}^{\mathrm{t}}\mathbf{A}_{i}\mathbf{P},\,i=1,2,3$.
Then $\mathbf{Y}_{i},\,i=1,2,3$ have the same eigenvalues $\{0,1/5,3/5\}$.
It is easily seen that $\mathbf{A}_{[\omega]_{m}}^{\mathrm{t}}\mathbf{P}\mathbf{A}_{[\omega]_{m}}=\mathbf{Y}_{[\omega]_{m}}^{\mathrm{t}}\mathbf{Y}_{[\omega]_{m}}$
for every $m\in\mathbb{N}_{+}$ and every $\omega\in\mathrm{W}_{m}$,
where $\mathbf{Y}_{[\omega]_{m}}=\mathbf{Y}_{\omega_{m}}\cdots\mathbf{Y}_{\omega_{2}}\mathbf{Y}_{\omega_{1}}$.
Therefore, $\mathrm{tr}\big(\mathbf{Y}_{[\omega]_{m}}^{\mathrm{t}}\mathbf{Y}_{[\omega]_{m}}\big)\le C_{\ast}\,(3/5)^{2m}$,
which implies that
\[
\sup_{\omega\in\mathrm{W}_{\ast}}\lim_{m\to\infty}\big[\mathrm{tr}\big(\mathbf{Y}_{[\omega]_{m}}^{\mathrm{t}}\mathbf{Y}_{[\omega]_{m}}\big)\big]^{1/m}\le(3/5)^{2}.
\]
For the reverse, let $\omega=111\dots\,\in\mathrm{W}_{\ast}$. Then
$\lim_{m\to\infty}\big[\mathrm{tr}\big(\mathbf{Y}_{[\omega]_{m}}^{\mathrm{t}}\mathbf{Y}_{[\omega]_{m}}\big)\big]^{1/m}=(3/5)^{2}$.
This proves \eqref{eq:-23}.
\end{proof}
\begin{rem}
\label{rem:-3}Setting $p=1,\;q=2$ in \eqref{eq:-44} gives the Nash
inequality on $\hat{\mathbb{S}}$ (see \citep[Theorem 4.1]{FHK94})
\[
\Vert u\Vert_{L^{2}(\hat{\nu})}^{2+4/d_{s}}\le C\,\hat{\mathcal{E}}(u)\Vert u\Vert_{L^{1}(\hat{\nu})}^{4/d_{s}},\;\;u\in\mathcal{F}(\hat{\mathbb{S}}).
\]
\end{rem}
Conclusions similar to that of Theorem \ref{thm:-2} hold when the
roles of $\hat{\sigma}$ and $\hat{\nu}$ are exchanged. More specifically,
let $\hat{\sigma}$ be a Borel measure on $\hat{\mathbb{S}}$ satisfying
the following condition: there exist constants $C_{\hat{\sigma}}\ge1$
and $0<\underbar{\ensuremath{\delta}}\le\bar{\delta}<\infty$ such
that
\begin{equation}
\left\{ \begin{aligned}C_{\hat{\sigma}}^{-1}\,\hat{\nu}(S)^{1/\underbar{\ensuremath{{\scriptstyle \delta}}}}\le\hat{\sigma}(S), & \;\;\text{if}\;\ 0<\text{diam}(S)<1,\\
C_{\hat{\sigma}}^{-1}\,\hat{\nu}(S)^{1/\bar{\delta}}\le\hat{\sigma}(S), & \;\;\text{if}\;\;\text{diam}(S)\ge1,
\end{aligned}
\right.\tag{{\text{M.1'}}}\label{eq:-11}
\end{equation}
for any dyadic triangle $S\subseteq\hat{\mathbb{S}}$. For measures
$\hat{\sigma}$ satisfying \eqref{eq:-11}, we have Theorem \ref{thm:-4}
and Theorem \ref{thm:-3} below, of which the proofs will be omitted
as they are are similar to those of Theorem \ref{thm:-2} and Theorem
\ref{thm:-1}.
\begin{thm}
\label{thm:-4}Let $1\le p\le q\le\infty,\;q\ge2$. Suppose that $\hat{\sigma}$
is a Borel measure on $\hat{\mathbb{S}}$ satisfying the condition
\eqref{eq:-11}. Then
\begin{equation}
\Vert u\Vert_{L^{q}(\hat{\nu})}\le C\,\sum_{i=1,2}\hat{\mathcal{E}}(u)^{a_{i}/2}\Vert u\Vert_{L^{p}(\hat{\sigma})}^{1-a_{i}},\;\;u\in\mathcal{F}(\hat{\mathbb{S}}),\label{eq:-41}
\end{equation}
where
\begin{equation}
a_{1}=\Big[\frac{1/(p\underbar{\ensuremath{\delta}})-1/q}{1/(p\underbar{\ensuremath{\delta}})+1/(2\delta_{s})}\Big]^{+},\;a_{2}=\Big[\frac{1/(p\bar{\delta})-1/q}{1/(p\bar{\delta})+1/(2\delta_{s})}\Big]^{+},\label{eq:-42}
\end{equation}
and $C>0$ is a constant depending only on the constant $C_{\hat{\sigma}}$
in \eqref{eq:-11}.

Moreover, if there exists a sequence $\{S_{m}\}_{m\in\mathbb{Z}}$
of dyadic triangles such that 
\[
\lim_{m\to-\infty}\mathrm{diam}(S_{m})=0,\;\lim_{m\to\infty}\mathrm{diam}(S_{m})=\infty
\]
and
\begin{equation}
1/\underbar{\ensuremath{\delta}}=\lim_{m\to-\infty}\frac{\log\hat{\sigma}(S_{m})}{\log\hat{\nu}(S_{m})},\;1/\bar{\delta}=\lim_{m\to\infty}\frac{\log\hat{\sigma}(S_{m})}{\log\hat{\nu}(S_{m})},\tag{{\text{M.2'}}}\label{eq:-9}
\end{equation}
then the pair of exponents given by \eqref{eq:-42} is optimal in
the following sense: if
\begin{equation}
\Vert u\Vert_{L^{q}(\hat{\nu})}\le C\,\sum_{i=1}^{N}\hat{\mathcal{E}}(u)^{b_{i}/2}\Vert u\Vert_{L^{p}(\hat{\sigma})}^{1-b_{i}},\;\;u\in\mathcal{F}(\hat{\mathbb{S}}),\label{eq:-46}
\end{equation}
for some constants $b_{i}\in[0,1]$ where $1\le i\le N,\,N\in\mathbb{N}_{+}$
and $C>0$ independent of $u$, then $\min_{i}b_{i}\le a_{2}\le a_{1}\le\max_{i}b_{i}.$
\end{thm}
\begin{rem}
\label{rem:-2}Theorem \ref{thm:-2} and Theorem \ref{thm:-4} can
be easily combined to yield the following
\begin{equation}
\Vert u\Vert_{L^{q}(\hat{\sigma}_{2})}\le C\,\sum_{i=1,2}\hat{\mathcal{E}}(u)^{a_{i}/2}\Vert u\Vert_{L^{p}(\hat{\sigma}_{1})}^{1-a_{i}},\;\;u\in\mathcal{F}(\hat{\mathbb{S}}),\label{eq:-26}
\end{equation}
where $\hat{\sigma}_{2}$ satisfies \eqref{eq:-47} with $\bar{\delta}=\bar{\delta}_{2},\;\underbar{\ensuremath{\delta}}=\underbar{\ensuremath{\delta}}_{2}$,
$\hat{\sigma}_{1}$ satisfies \eqref{eq:-11} with $\bar{\delta}=\bar{\delta}_{1},\;\underbar{\ensuremath{\delta}}=\underbar{\ensuremath{\delta}}_{1}$,
and
\[
a_{1}=\Big[\frac{1/(p\underbar{\ensuremath{\delta}}_{1})-1/(q\bar{\delta}_{2})}{1/(p\underbar{\ensuremath{\delta}}_{1})+1/(2\delta_{2})}\Big]^{+},\;a_{2}=\Big[\frac{1/(p\bar{\delta}_{1})-1/(q\underbar{\ensuremath{\delta}}_{2})}{1/(p\bar{\delta}_{1})+1/(2\delta_{s})}\Big]^{+}.
\]
\end{rem}
\begin{thm}
\label{thm:-3}Let $\hat{\sigma}$ be a Borel measure on $\hat{\mathbb{S}}$.
Suppose that there exist some constants $p,q\in(0,\infty)$, $b_{i}\in[0,1]$
where $1\le i\le N$ and $C>0$ such that \eqref{eq:-46} holds for
all $u\in\mathcal{F}(\hat{\mathbb{S}})$. Then there exist constants
$0<\underbar{\ensuremath{\delta}}\le\bar{\delta}\le\infty$ such that
the condition \eqref{eq:-11} is satisfied.
\end{thm}
\begin{cor}
\label{cor:-2}The inequality \eqref{eq:-41} holds with $\hat{\sigma}=\hat{\mu}$.
The pair $(a_{1},a_{2})$ of exponents given by \eqref{eq:-42} is
optimal, where the constants $\bar{\delta}=1$ and $\underbar{\ensuremath{\delta}}$
is given by $1/\underbar{\ensuremath{\delta}}=1/\delta_{s}+2.$
\end{cor}
\begin{rem}
\label{rem:-4}The value of $\underbar{\ensuremath{\delta}}$ in Corollary
\ref{cor:-2} follows from the fact that
\[
\inf_{\omega\in\mathrm{W}_{\ast}}\lim_{m\to\infty}\big[\mathrm{tr}\big(\mathbf{A}_{[\omega]_{m}}^{\mathrm{t}}\mathbf{P}\mathbf{A}_{[\omega]_{m}}\big)\big]^{1/m}=3/25,
\]
which will be given in another work by the present authors.
\end{rem}
We end this section with the corresponding Sobolev inequalities on
the compact gasket $\mathbb{S}$, whose proof shall be omitted. Let
$\sigma$ be a finite Borel measure on $\mathbb{S}$. For the compact
gasket, only the first part of the condition \eqref{eq:-47} is relevant,
i.e.
\begin{equation}
\sigma\big(\mathbb{S}_{[\omega]_{m}}\big)\le C_{\sigma}\nu\big(\mathbb{S}_{[\omega]_{m}}\big)^{1/\bar{\delta}},\;\;\text{for all}\ \omega\in\mathrm{W}_{\ast}\ \text{and all}\ m\in\mathbb{N},\label{eq:-49}
\end{equation}
where $C_{\sigma}>0$ and $\bar{\delta}\in[1,\infty]$ are constants
depending only on the Borel measure $\sigma$. Similarly, we only
need the first part of the condition \eqref{eq:-11}, i.e.
\begin{equation}
C_{\sigma}^{-1}\nu\big(\mathbb{S}_{[\omega]_{m}}\big)^{1/\underbar{\ensuremath{{\scriptstyle \delta}}}}\le\sigma\big(\mathbb{S}_{[\omega]_{m}}\big),\;\;\text{for all}\ \omega\in\mathrm{W}_{\ast}\ \text{and all}\ m\in\mathbb{N},\label{eq:-53}
\end{equation}
where $\underbar{\ensuremath{\delta}}\in(0,\infty]$ is a constant
depending only on the Borel measure $\sigma$.
\begin{thm}
\label{thm:-5}Let $1\le p\le q\le\infty,\;q\ge2$, and let $\sigma$
be a finite Borel measure on $\mathbb{S}$.\\
\\
(a) Suppose that $\sigma$ satisfies \eqref{eq:-49}. Then for any
$u\in\mathcal{F}(\mathbb{S})$,
\begin{equation}
\Vert u-c\Vert_{L^{q}(\sigma)}\le C\,\mathcal{E}(u)^{a/2}\Vert u-c\Vert_{L^{p}(\nu)}^{1-a}.\label{eq:-52}
\end{equation}
where $c$ is any constant satisfying $\min_{\mathbb{S}}u\le c\le\max_{\mathbb{S}}u$,
and
\begin{equation}
a=\Big[\frac{1/p-1/(q\bar{\delta})}{1/p+1/(2\delta_{s})}\Big]^{+},\label{eq:-51}
\end{equation}
and $C>0$ is a constant depending only on the constant $C_{\sigma}$
in \eqref{eq:-49}. Therefore, for any $u\in\mathcal{F}(\mathbb{S})$,
\begin{equation}
\Vert u\Vert_{L^{q}(\sigma)}\le C\,\big[\mathcal{E}(u)^{a/2}\Vert u\Vert_{L^{p}(\nu)}^{1-a}+\Vert u\Vert_{L^{p}(\nu)}\big].\label{eq:-50}
\end{equation}
Moreover, the exponent $a$ given by \eqref{eq:-51} is optimal in
the sense that if \eqref{eq:-52} holds for some $a\in[0,1]$, then
$a\ge\Big[\frac{1/p-1/(q\bar{\delta})}{1/p+1/(2\delta_{s})}\Big]^{+}.$\\
\\
(b) Suppose that $\sigma$ satisfies \eqref{eq:-53}. Then the conclusions
of (a) hold when $\sigma$ and $\nu$ are exchanged and the exponent
\eqref{eq:-51} is replaced by
\[
a=\Big[\frac{1/(p\underbar{\ensuremath{\delta}})-1/q}{1/(p\underbar{\ensuremath{\delta}})+1/(2\delta_{s})}\Big]^{+},
\]
\end{thm}
\begin{rem}
\label{rem:-5}Setting $\sigma=\nu,\,\underbar{\ensuremath{\delta}}=\bar{\delta}=1$
and $p=1,\,q=2$ in \eqref{eq:-50} gives the Nash inequality on $\mathbb{S}$
(see \citep[Theorem 4.4]{FHK94} or \citep[Theorem 5.3.3]{Ki01})
\[
\Vert u\Vert_{L^{2}(\nu)}^{2+4/d_{s}}\le C\big[\mathcal{E}(u)+\Vert u\Vert_{L^{1}(\nu)}^{2}\big]\,\Vert u\Vert_{L^{1}(\nu)}^{4/d_{s}}\le C\big[\mathcal{E}(u)+\Vert u\Vert_{L^{2}(\nu)}^{2}\big]\,\Vert u\Vert_{L^{1}(\nu)}^{4/d_{s}},\;\;u\in\mathcal{F}(\mathbb{S}).
\]
\end{rem}
\begin{cor}
\label{cor:}For any $u\in\mathcal{F}(\mathbb{S})$,
\[
\Vert u\Vert_{L^{2}(\mu)}\le C\,\big[\mathcal{E}(u)^{(d_{s}-1)/2}\Vert u\Vert_{L^{2}(\nu)}^{2-d_{s}}+\Vert u\Vert_{L^{2}(\nu)}\big].
\]
If $u\in\mathcal{F}(\mathbb{S}\backslash\mathrm{V}_{0})$ in addition,
then by \eqref{eq:-52} with $c=0$,
\begin{equation}
\Vert u\Vert_{L^{2}(\mu)}\le C\,\mathcal{E}(u)^{(d_{s}-1)/2}\Vert u\Vert_{L^{2}(\nu)}^{2-d_{s}}.\label{eq:-25}
\end{equation}
\end{cor}

\section{\label{sec:-1}Semi-linear parabolic PDEs}

In this section, we study a type of semi-linear parabolic equations
on $\mathbb{S}$, for which energy estimates and existence and uniqueness
of solutions are established (Theorem \ref{thm:}). Moreover, the
regularity of solutions to these PDEs is derived under additional
conditions.

We consider the following initial-boundary value problem for semi-linear
parabolic PDEs (see Definition \ref{def:} below for a precise interpretation)
\begin{equation}
\left\{ \begin{aligned}\partial_{t}u & \,d\nu=\mathcal{L}u\,d\nu+f(t,x,u,\nabla u)d\mu,\;\;\text{in}\,(0,T]\times(\mathbb{S}\backslash\mathrm{V}_{0}),\\
u & =0\;\;\text{on}\ (0,T]\times\mathrm{V}_{0},\;\;u(0)=\psi,
\end{aligned}
\right.\label{eq:}
\end{equation}
where $\psi\in L^{2}(\nu)$, and the coefficient $f:[0,T]\times\mathbb{S}\times\mathbb{R}^{2}\to\mathbb{R}$
satisfies the following:

\ (i) There exists a constant $K>0$ such that

\begin{equation}
\vert f(t,x,y,z)-f(t,x,\bar{y},\bar{z})\vert\le K\,(\vert y-\bar{y}\vert+\vert z-\bar{z}\vert),\tag{{\text{A.1}}}\label{eq:-1}
\end{equation}
for all $(t,x)\in[0,T]\times\mathbb{S},\,(y,z),(\bar{y},\bar{z})\in\mathbb{R}^{2}$;

\smallskip{}

(ii) $f(\cdot,0,0)\in L^{2}(0,T;L^{2}(\mu))$, that is,
\begin{equation}
\Vert f(\cdot,0,0)\Vert_{L^{2}(0,T;L^{2}(\mu))}^{2}=\int_{0}^{T}\int_{\mathbb{S}}f(t,x,0,0)^{2}\mu(dx)dt<\infty.\tag{{\text{A.2}}}\label{eq:-2}
\end{equation}

\begin{rem}
\label{rem:-10}There exist different formulations of non-linear PDEs
on fractals. For example, a type of non-linear equations on fractals
was considered by in \citep{HRT13}, where the non-linearity $f(\nabla u)$
is a bounded mapping $f:L^{2}(\mu)\to L^{2}(\nu)$. The equations
studied there are essentially defined via a single measure (the Hausdorff
measure $\nu$). Therefore, the PDEs studied in this paper are different
in essence from those considered in \citep{HRT13} in the way the
gradients interact with the equations.
\end{rem}
From now on, we shall use the notation $\langle f,g\rangle_{\lambda}=\int_{\mathbb{S}}fg\,d\lambda$
for any Borel measure $\lambda$ on $\mathbb{S}$ and any $\lambda\text{-}\text{a.e.}$
defined functions $f,\,g$ on $\mathbb{S}$, whenever the integral
is well-defined. As in the previous section, we denote by $C_{\ast}$
a generic universal constant which may vary on different occasions.

Let $\{P_{t}\}_{t\ge0}$ be the Markov semigroup associated with the
killed Brownian motion on $\mathbb{S}$, the diffusion processes associated
with the Dirichlet form $(\mathcal{E},\mathcal{F}(\mathbb{S}\backslash\mathrm{V}_{0}))$.
$\{P_{t}\}_{t\ge0}$ admits a jointly continuous heat kernel $p(t,x,y)$,
which is $C^{\infty}$ in $t$ (cf. \citep[Theorem 1.5]{BP88}). The
following result on heat kernel and resolvent kernel estimate was
first proved in \citep[Theorem 1.5 and Theorem 1.8]{BP88}.
\begin{lem}
\label{lem:-2}For each $t>0$ 
\[
p(t,x,y)\le C_{\ast}\,t^{-d_{s}/2},\;\;x,y\in\mathbb{S},
\]
is valid. Let $\rho_{\alpha},\,\alpha>0$ be the $\alpha$-resolvent
kernel of $\mathcal{L}$, that is,
\[
\rho_{\alpha}(x,y)=\int_{0}^{\infty}e^{-\alpha t}p(t,x,y)\,dt,\;\;x,y\in\mathbb{S}.
\]
Then $\rho_{\alpha}(\cdot,\cdot)$ is Lipschitz continuous with respect
to the resistance metric, i.e.
\[
|\rho_{\alpha}(x,z)-\rho_{\alpha}(y,z)|\le C_{\alpha}R(x,y),\;\;x,y,z\in\mathbb{S},
\]
for some constant $C_{\alpha}>0$ depending only on $\alpha$.
\end{lem}
In view of the joint continuity of $p(t,x,y)$, the definition below
is legitimate.
\begin{defn}
For any Radon measure $\lambda$ on $\mathbb{S}$, we define $P_{t}\lambda(x)=\int_{\mathbb{S}}p(t,x,y)\,\lambda(dy),\;x\in\mathbb{S},\;t\in(0,\infty)$.
\end{defn}
\begin{rem}
\label{rem:-7}(i) Let $\lambda$ be a Radon measure on $\mathbb{S}$.
By the symmetry of $p(t,\cdot,\cdot)$, it is easy to see that $\langle P_{t}(g\lambda),f\rangle_{\nu}=\langle g,P_{t}f\rangle_{\lambda}$
for all $f\in L^{2}(\nu),\,g\in L^{1}(\lambda)$.

\ (ii) For any Radon measure $\lambda$ on $\mathbb{S}$, we have
$P_{t}\lambda\in\mathrm{Dom}(\mathcal{L})$ for $t>0$. In fact, since
$p(t,x,y)\in C((0,\infty)\times\mathbb{S}\times\mathbb{S})$, we have
$P_{t/2}\lambda\in C(\mathbb{S})$, which implies that $P_{t}\lambda=P_{t/2}(P_{t/2}\lambda)\in\mathrm{Dom}(\mathcal{L})$.
Moreover, $P_{t}\lambda\in C^{1}(0,\infty;L^{2}(\nu))$ and $\frac{d}{dt}P_{t}\lambda=\mathcal{L}P_{t}\lambda$.

(iii) Notice that, due to the singularity between $\nu$ and $\mu$,
the contractivity $\Vert P_{t}(g\mu)\Vert_{L^{2}(\nu)}\le\Vert g\Vert_{L^{2}(\mu)},\,t>0$
is no longer valid in general. In fact, for $g\in L^{2}(\mu),\,g\not=0$,
we have
\[
\lim_{t\to0}\Vert P_{t}(g\mu)\Vert_{L^{2}(\nu)}=\infty.
\]
To see this, suppose contrarily that $\lim_{t\to0}\Vert P_{t}(g\mu)\Vert_{L^{2}(\nu)}=\sup_{t>0}\Vert P_{t}(g\mu)\Vert_{L^{2}(\nu)}<\infty$.
Then there exists a unique $g_{0}\in L^{2}(\nu)$ such that $\lim_{t\to0}P_{t}(g\mu)=g_{0}$
weakly in $L^{2}(\nu)$. On the other hand, for any $v\in\mathcal{F}(\mathbb{S}\backslash\mathrm{V}_{0})$,
we have 
\[
\langle g_{0},v\rangle_{\nu}=\lim_{t\to0}\langle P_{t}(g\mu),v\rangle_{\nu}=\lim_{t\to0}\langle g,P_{t}v\rangle_{\mu}=\langle g,v\rangle_{\mu},
\]
where the last equality follows from the uniform convergence $\lim_{t\to0}P_{t}v=v$
as a consequence of the convergence in $\mathcal{F}(\mathbb{S}\backslash\mathrm{V}_{0})$.
By the density of $\mathcal{F}(\mathbb{S}\backslash\mathrm{V}_{0})$
in $C(\mathbb{S})$, it is seen that $g\mu=g_{0}\nu$, which contradicts
the fact that $\nu$ and $\mu$ are mutually singular.
\end{rem}
To study the semi-linear parabolic PDEs \eqref{eq:}, let us first
investigate the formal integral
\begin{equation}
\int_{0}^{t}P_{t-s}(g(s)\mu)\,ds,\label{eq:-80}
\end{equation}
which is the formal solution to the equation $\partial_{t}u\,d\nu=\mathcal{L}u\,d\nu+g(t)\,d\mu$.
Since $P_{t}$ is not bounded from $L^{2}(\mu)$ to $L^{2}(\nu)$
(cf. Remark \ref{rem:-7}.(iii) above), there is a singularity in
the integrand of \eqref{eq:-80} at $s=t$. We shall show that \eqref{eq:-80}
is a well-defined function in the space $L^{\infty}(0,T;L^{2}(\nu))\cap L^{2}(0,T;\mathcal{F}(\mathbb{S}\backslash\mathrm{V}_{0}))$,
and is jointly Hölder continuous if $g(t)$ is uniformly bounded in
$L^{2}(\mu)$. To formulate the results, it is convenient to introduce
several definitions.
\begin{defn}
\label{def:-2}For any $v\in L^{2}(\nu)$, define
\[
\Vert v\Vert_{\mathcal{F}^{-1}}=\sup\big\{\langle u,v\rangle_{\nu}:u\in\mathcal{F}(\mathbb{S}),\,\Vert u\Vert_{\mathcal{F}}\le1\big\},
\]
where 
\[
\Vert u\Vert_{\mathcal{F}}=\big[\Vert u\Vert_{L^{2}(\nu)}^{2}+\mathcal{E}(u)\big]^{1/2}.
\]
The space $\mathcal{F}^{-1}(\mathbb{S})$ is defined to be the $\Vert\cdot\Vert_{\mathcal{F}^{-1}}$-completion
of $L^{2}(\nu)$.
\end{defn}
\begin{defn}
\label{def:-3}Let $u\in L^{2}(0,T;\mathcal{F}(\mathbb{S}))$; that
is, $\int_{0}^{T}\mathcal{E}_{1}(u(t))\,dt<\infty$. An  ($\mathcal{F}(\mathbb{S})$-valued)
function $u$ is said to have a\emph{ weak derivative }$\partial_{t}u$\emph{
}in\emph{ }$L^{2}(0,T;\mathcal{F}^{-1}(\mathbb{S}))$, if $\partial_{t}u$
is an $\mathcal{F}^{-1}(\mathbb{S})$-valued function on $[0,T]$
satisfying
\[
\Big(\int_{0}^{T}\Vert\partial_{t}u(t)\Vert_{\mathcal{F}^{-1}}^{2}\,dt\Big)^{1/2}<\infty\;\;\text{and}\;\int_{0}^{T}\langle u(t),\partial_{t}v(t)\rangle_{\nu}\,dt=-\int_{0}^{T}\langle\partial_{t}u(t),v(t)\rangle_{\nu}\,dt
\]
 for all $v\in C^{1}(0,T;\mathcal{F}(\mathbb{S}))$ with $v(0)=v(T)=0$.
\end{defn}
The following lemma can be easily shown by a mollifier argument similar
to that of \citep[Theorem 3, Section 5.9]{Eva10}.
\begin{lem}
\label{lem:}Suppose that $u\in L^{2}(0,T;\mathcal{F}(\mathbb{S}\backslash\mathrm{V}_{0}))$
has a weak derivative $\partial_{t}u\in L^{2}(0,T;\mathcal{F}^{-1}(\mathbb{S}))$.
Then

\smallskip{}

\noindent (a) $u\in C(0,T;L^{2}(\nu))$;

\smallskip{}

\noindent (b) The function $t\mapsto\Vert u(t)\Vert_{L^{2}(\nu)}^{2}$
is absolutely continuous, and $\frac{d}{dt}\Vert u(t)\Vert_{L^{2}(\nu)}^{2}=2\langle\partial_{t}u(t),u(t)\rangle_{\nu}$
for a.e. $t\in[0,T]$.
\end{lem}
We derive properties of the convolution \eqref{eq:-80} in the following
lemmas.
\begin{lem}
\label{lem:-3}Let $g\in L^{2}(0,T;L^{2}(\mu))$. For each $\delta\in(0,T)$,
let
\begin{equation}
u_{\delta}(t)=\int_{0}^{(t-\delta)^{+}}P_{t-s}(g(s)\mu)\,ds,\;\;t\in[0,T].\label{eq:-37}
\end{equation}
Then $u_{\delta}\in L^{\infty}(0,T;L^{2}(\nu))\cap L^{2}(0,T;\mathcal{F}(\mathbb{S}\backslash\mathrm{V}_{0}))$
and

\[
\Vert u_{\delta}\Vert_{L^{\infty}(0,T;L^{2}(\nu))}^{2}+\int_{0}^{T}e^{C_{\epsilon}(T-s)}\mathcal{E}(u_{\delta}(s))\,ds\le\epsilon\,\int_{0}^{T}e^{C_{\epsilon}(T-s)}\Vert g(s)\Vert_{L^{2}(\mu)}^{2}\,ds,
\]
for any $\epsilon>0$, where $C_{\epsilon}>0$ is a constant depending
only on $\epsilon$. Moreover,
\[
\Vert\partial_{t}u_{\delta}\Vert_{L^{2}(0,T;\mathcal{F}^{-1})}\le C_{\ast}\,\Vert g\Vert_{L^{2}(0,T;L^{2}(\mu))}.
\]
\end{lem}
\begin{proof}
It is convenient to set $g(t)=0$ for $t<0$. Clearly, $u_{\delta}(t)\in\mathrm{Dom}(\mathcal{L}),\;t\in[0,T]$.
For each $s\in(0,T)$, since $t\mapsto P_{t-s}(g(s)\mu)=P_{t-s-\delta}[P_{\delta}(g(s)\mu)],\;t\in(s+\delta,T)$
is a continuously differentiable $L^{2}(\nu)$-valued function, we
see that $u_{\delta}\in C^{1}(\delta,T;L^{2}(\nu))$ and
\begin{equation}
\partial_{t}u_{\delta}(t)=P_{\delta}(g(t-\delta)\mu)+\int_{0}^{t-\delta}\mathcal{L}P_{t-s}(g(s)\mu)\,ds=\mathcal{L}u_{\delta}(t)+P_{\delta}(g(t-\delta)\mu).\label{eq:-34}
\end{equation}
For any $\epsilon>0$ and each $t\in(0,T)$, testing \eqref{eq:-34}
against $u_{\delta}$ and applying Corollary \ref{cor:} gives that
\[
\begin{aligned}\frac{1}{2}\frac{d}{dt}\Vert u_{\delta}(t)\Vert_{L^{2}(\nu)}^{2} & =\langle P_{\delta}(g(t-\delta)\mu),u_{\delta}(t)\rangle_{\nu}-\mathcal{E}(u_{\delta}(t))\\
\vphantom{\frac{1}{2}} & =\langle g(t-\delta),P_{\delta}(u_{\delta}(t))\rangle_{\mu}-\mathcal{E}(u_{\delta}(t))\\
\vphantom{\frac{1}{2}} & \le C_{\epsilon}\,\mathcal{E}[P_{\delta}(u_{\delta}(t))]^{d_{s}-1}\Vert P_{\delta}(u_{\delta}(t))\Vert_{L^{2}(\nu)}^{2(2-d_{s})}-\mathcal{E}(u_{\delta}(t))+\epsilon\Vert g(t-\delta)\Vert_{L^{2}(\mu)}^{2}\\
\vphantom{\frac{1}{2}} & \le C_{\epsilon}\,\mathcal{E}(u_{\delta}(t))^{d_{s}-1}\Vert u_{\delta}(t)\Vert_{L^{2}(\nu)}^{2(2-d_{s})}-\mathcal{E}(u_{\delta}(t))+\epsilon\Vert g(t-\delta)\Vert_{L^{2}(\mu)}^{2}\\
\vphantom{\frac{1}{2}} & \le C_{\epsilon}\,\Vert u_{\delta}(t)\Vert_{L^{2}(\nu)}^{2}-\frac{1}{2}\mathcal{E}(u_{\delta}(t))+\epsilon\Vert g(t-\delta)\Vert_{L^{2}(\mu)}^{2},
\end{aligned}
\]
where $C_{\epsilon}>0$ denotes a generic constant depending only
on $\epsilon$ which may vary on different occasions. By Grönwall's
inequality and the fact that $u_{\delta}(t)=0,\,t\in[0,\delta]$,
we deduce
\begin{equation}
\Vert u_{\delta}(t)\Vert_{L^{2}(\nu)}^{2}+\int_{0}^{t}e^{C_{\epsilon}(t-s)}\mathcal{E}(u_{\delta}(s))\,ds\le\epsilon\,\int_{0}^{(t-\delta)^{+}}e^{C_{\epsilon}(t-s)}\Vert g(s)\Vert_{L^{2}(\mu)}^{2}\,ds.\label{eq:-33}
\end{equation}
By \eqref{eq:-34} again, for any $v\in\mathcal{F}(\mathbb{S}\backslash\mathrm{V}_{0})$,
\[
\begin{aligned}\vert\langle\partial_{t}u_{\delta}(t),v\rangle_{\nu}\vert & \le\big|\langle g(t-\delta),P_{\delta}v\rangle_{\mu}\big|+\mathcal{E}(u_{\delta}(t))^{1/2}\mathcal{E}(v)^{1/2}\\
 & \le C_{\ast}\,\Vert g(t-\delta)\Vert_{L^{2}(\mu)}\mathcal{E}(P_{\delta}v)^{1/2}+\mathcal{E}(u_{\delta}(t))^{1/2}\mathcal{E}(v)^{1/2}\\
 & \le C_{\ast}\,\big[\Vert g(t-\delta)\Vert_{L^{2}(\mu)}+\mathcal{E}(u_{\delta}(t))^{1/2}\big]\Vert v\Vert_{\mathcal{F}}.
\end{aligned}
\]
The above inequality also holds for $v\in\mathcal{F}(\mathbb{S})$.
This can be seen by considering the $\mathcal{F}$-orthogonal projection
of $v$ on $\mathcal{F}(\mathbb{S}\backslash\mathrm{V}_{0})$. Therefore,
\[
\Vert\partial_{t}u_{\delta}(t)\Vert_{\mathcal{F}^{-1}}\le C_{\ast}\,\big[\Vert g(t-\delta)\Vert_{L^{2}(\mu)}+\mathcal{E}(u_{\delta}(t))^{1/2}\big],\;\;t\in(\delta,T],
\]
which, together with \eqref{eq:-33}, implies the desired estimate
for $\Vert\partial_{t}u_{\delta}\Vert_{L^{2}(0,T;\mathcal{F}^{-1})}$.
\end{proof}
\begin{lem}
\label{lem:-4}The limit
\begin{equation}
u(t)=\lim_{\delta\to0}\int_{0}^{(t-\delta)^{+}}P_{t-s}(g(s)\mu)\,ds,\label{eq:-35}
\end{equation}
 exists with respect to the norm $\Vert\cdot\Vert_{L^{\infty}(0,T;L^{2}(\nu))}+\Vert\cdot\Vert_{L^{2}(0,T;\mathcal{F})}$,
and satisfies
\[
\Vert u\Vert_{L^{\infty}(0,T;L^{2}(\nu))}^{2}+\int_{0}^{T}e^{C_{\epsilon}(T-s)}\mathcal{E}(u(s))\,ds\le\epsilon\,\int_{0}^{T}e^{C_{\epsilon}(T-s)}\Vert g(s)\Vert_{L^{2}(\mu)}^{2}\,ds.
\]
Moreover, $u(t)$ has a weak derivative $\partial_{t}u$ in $L^{2}(0,T;\mathcal{F}^{-1})$,
and
\[
\Vert\partial_{t}u\Vert_{L^{2}(0,T;\mathcal{F}^{-1})}\le C_{\ast}\,\Vert g\Vert_{L^{2}(0,T;L^{2}(\mu))}.
\]
\end{lem}
\begin{proof}
As before, we set $g(t)=0$ for $t<0$. Let $\delta,\delta^{\prime}\in(0,T)$
and $w=u_{\delta}-u_{\delta^{\prime}}$, where $u_{\delta}$ are the
functions defined by \eqref{eq:-37}. By \eqref{eq:-34}, we have
\[
\partial_{t}w=\mathcal{L}w+P_{\delta}[(g(t-\delta)-g(t-\delta^{\prime}))\mu]+(P_{\delta}-P_{\delta^{\prime}})(g(t-\delta^{\prime})\mu),
\]
from which it follows that
\begin{equation}
\begin{aligned}\frac{1}{2}\frac{d}{dt}\Vert w(t)\Vert_{L^{2}(\nu)}^{2} & =-\mathcal{E}(w(t))+\langle P_{\delta}[(g(t-\delta)-g(t-\delta^{\prime}))\mu],w(t)\rangle_{\nu}\\
 & \quad+\langle(P_{\delta}-P_{\delta^{\prime}})(g(t-\delta^{\prime})\mu),w(t)\rangle_{\nu}.
\end{aligned}
\label{eq:-31}
\end{equation}
The first term on the right hand side of \eqref{eq:-31} can be estimated
in the same way as in the proof of Lemma \ref{lem:-3}, which yields
that
\[
\langle P_{\delta}[(g(t-\delta)-g(t-\delta^{\prime}))\mu],w(t)\rangle_{\nu}\le\frac{1}{2}\Vert g(t-\delta)-g(t-\delta^{\prime})\Vert_{L^{2}(\mu)}^{2}+\frac{1}{2}\mathcal{E}(w(t))^{d_{s}-1}\Vert w(t)\Vert_{L^{2}(\nu)}^{2(2-d_{s})}.
\]
For the second term on the right hand side of \eqref{eq:-31}, we
have
\[
\langle(P_{\delta}-P_{\delta^{\prime}})(g(t-\delta^{\prime})\mu),w(t)\rangle_{\nu}\le C_{\ast}\mathcal{E}((P_{\delta}-P_{\delta^{\prime}})w(t))^{(d_{s}-1)/2}\Vert(P_{\delta}-P_{\delta^{\prime}})w(t)\Vert_{L^{2}(\nu)}^{2-d_{s}}\Vert g(t-\delta^{\prime})\Vert_{L^{2}(\mu)}.
\]
By the spectral decomposition,
\[
\begin{aligned}\Vert(P_{\delta}-P_{\delta^{\prime}})w(t)\Vert_{L^{2}(\nu)}^{2} & =\int_{0}^{\infty}(e^{-\lambda\delta}-e^{-\lambda\delta^{\prime}})^{2}d\Vert E_{\lambda}w(t)\Vert_{L^{2}(\nu)}^{2}\\
 & \le\int_{0}^{\infty}(1-e^{-\lambda|\delta-\delta^{\prime}|})^{2}d\Vert E_{\lambda}w(t)\Vert_{L^{2}(\nu)}^{2}\\
 & \le|\delta-\delta^{\prime}|\int_{0}^{\infty}\lambda d\Vert E_{\lambda}w(t)\Vert_{L^{2}(\nu)}^{2}\\
\vphantom{\int} & =|\delta-\delta^{\prime}|\mathcal{E}(w(t)),
\end{aligned}
\]
which, together with the fact that $\mathcal{E}((P_{\delta}-P_{\delta^{\prime}})w(t))\le\mathcal{E}(w(t))$,
implies that
\[
\langle(P_{\delta}-P_{\delta^{\prime}})(g(t-\delta^{\prime})\mu),w(t)\rangle_{\nu}\le C_{\ast}|\delta-\delta^{\prime}|^{1-d_{s}/2}\mathcal{E}(w(t))^{1/2}\Vert g(t-\delta^{\prime})\Vert_{L^{2}(\mu)}.
\]
Therefore, we deduce from \eqref{eq:-31} that
\[
\begin{aligned}\frac{d}{dt}\Vert w(t)\Vert_{L^{2}(\nu)}^{2} & \le-\mathcal{E}(w(t))+C_{\ast}\Vert w(t)\Vert_{L^{2}(\nu)}^{2}+C_{\ast}\Vert g(t-\delta)-g(t-\delta^{\prime})\Vert_{L^{2}(\mu)}^{2}\\
 & \quad+C_{\ast}|\delta-\delta^{\prime}|^{2-d_{s}}\Vert g(t-\delta^{\prime})\Vert_{L^{2}(\mu)}^{2}.
\end{aligned}
\]
It follows from the above inequality and Grönwall's inequality that
\[
\begin{aligned}\Vert w\Vert_{L^{\infty}(0,T;L^{2}(\nu))}+\Vert w\Vert_{L^{2}(0,T;\mathcal{F})} & \le C_{\ast}\,\Big[|\delta-\delta^{\prime}|^{2-d_{s}}\Vert g\Vert_{L^{2}(0,T;L^{2}(\mu))}\\
 & \quad+\int_{0}^{T}\Vert g(t-\delta)-g(t-\delta^{\prime})\Vert_{L^{2}(\mu)}^{2}dt\Big].
\end{aligned}
\]
Therefore, $\{u_{\delta}\}$ is a Cauchy sequence with respect to
the norm $\Vert\cdot\Vert_{L^{\infty}(0,T;L^{2}(\nu))}+\Vert\cdot\Vert_{L^{2}(0,T;\mathcal{F})}$,
which proves the convergence of \eqref{eq:-35}. Moreover, the desired
estimates for $u$ follows readily from the similar estimates for
$u_{\delta}$.
\end{proof}
\begin{defn}
\label{def:-1}By virtue of Lemma \ref{lem:-4}, the convolution $\int_{0}^{t}P_{t-s}(g(s)\mu)\,ds$
can be defined to be the limit in \eqref{eq:-35}.
\end{defn}
\begin{lem}
\label{lem:-1}If $g\in L^{\infty}(0,T;L^{2}(\mu))$, then the convolution
$u$ defined by \eqref{eq:-35} is jointly continuous in $[0,T]\times\mathbb{S}$.
Moreover, for any $0<\theta<\frac{3}{2}(1-d_{s}/2)$,
\begin{equation}
|u(t,x)-u(s,y)|\le\Vert g\Vert_{L^{\infty}(0,T;L^{2}(\mu))}\big[C_{\theta}|t-s|^{\theta}+C_{T}\,R(x,y)^{1/2}\big],\label{eq:-65}
\end{equation}
where $C_{\theta}>0$ is a constant depending only on $\theta$, and
$C_{T}>0$ one depending only on $T$.
\end{lem}
\begin{rem}
\label{rem:-6}The authors believe that $1/2$ is the correct Hölder
exponent in $x\in\mathbb{S}$ for \eqref{eq:-35} in general, which
is suggested by the fact that a generic $u\in\mathcal{F}(\mathbb{S})$
has only $\frac{1}{2}$-Hölder continuity (cf. \eqref{eq:-55}). As
a matter of fact, the convolution \eqref{eq:-35} only has mild regularity
in general due to the singularity between $\nu$ and $\mu$ (cf. Remark
\ref{rem:-8}.(ii) below).
\end{rem}
\begin{proof}
Let $g(t)=0$ for $t<0$. We first show that
\begin{equation}
|u(t,x)-u(t,y)|\le C_{T}R(x,y)^{1/2},\;\;x,y\in\mathbb{S},\label{eq:-67}
\end{equation}
where $C_{T}>0$ is a constant depending only on $T$. Denote $p_{s,x}(y)=p(s,x,y)$.
By the definition of $u(t)$, we have
\begin{equation}
\begin{aligned}\begin{aligned}|u(t,x)-u(t,y)|\end{aligned}
 & =\Big|\int_{0}^{t}\langle g(t-s),p_{s,x}-p_{s,y}\rangle_{\mu}\,ds\Big|\\
 & \le\Vert g\Vert_{L^{\infty}(0,T,L^{2}(\mu))}\,\int_{0}^{t}\Vert p_{s,x}-p_{s,y}\Vert_{L^{2}(\mu)}\,ds.
\end{aligned}
\label{eq:-32}
\end{equation}
By the Sobolev inequality \eqref{eq:-25},
\[
\Vert p_{s,x}-p_{s,y}\Vert_{L^{2}(\mu)}\le C\,\mathcal{E}(p_{s,x}-p_{s,y})^{(d_{s}-1)/2}\Vert p_{s,x}-p_{s,y}\Vert_{L^{2}(\nu)}^{2-d_{s}}.
\]
Let $-\mathcal{L}=\int_{0}^{\infty}\lambda\,dE_{\lambda}$ be the
spectral representation. Then
\begin{equation}
\begin{aligned}\mathcal{E}(p_{s,x}-p_{s,y}) & =\mathcal{E}(P_{s/2}(p_{s/2,x}-p_{s/2,y}))\\
 & =\int_{0}^{\infty}\lambda e^{-\lambda s}\,d\Vert E_{\lambda}(p_{s/2,x}-p_{s/2,y})\Vert_{L^{2}(\nu)}^{2}\\
\vphantom{\int_{0}^{\infty}} & \le s^{-1}\Vert p_{s/2,x}-p_{s/2,y}\Vert_{L^{2}(\nu)}^{2},
\end{aligned}
\label{eq:-75}
\end{equation}
Therefore
\[
\Vert p_{s,x}-p_{s,y}\Vert_{L^{2}(\mu)}\le Cs^{-(d_{s}-1)/2}\Vert p_{s/2,x}-p_{s/2,y}\Vert_{L^{2}(\nu)}^{d_{s}-1}\Vert p_{s,x}-p_{s,y}\Vert_{L^{2}(\nu)}^{2-d_{s}},
\]
By the inequality above and Hölder's inequality,
\begin{align*}
\int_{0}^{t}\Vert p_{s,x}-p_{s,y}\Vert_{L^{2}(\mu)}ds & \le C\,\Big(\int_{0}^{t}s^{1-d_{s}}\,ds\Big)^{1/2}\Big(\int_{0}^{t}\Vert p_{s/2,x}-p_{s/2,y}\Vert_{L^{2}(\nu)}^{2}ds\Big)^{(d_{s}-1)/2}\\
 & \quad\times\Big(\int_{0}^{t}\Vert p_{s,x}-p_{s,y}\Vert_{L^{2}(\nu)}^{2}ds\Big)^{1-d_{s}/2}\\
 & \le C_{T}\Big(\int_{0}^{t}\Vert p_{s/2,x}-p_{s/2,y}\Vert_{L^{2}(\nu)}^{2}ds\Big)^{1/2},
\end{align*}
where we have used the fact that $p_{s,x}-p_{s,y}=P_{s/2}(p_{s/2,x}-p_{s/2,y})$
and the $L^{2}(\nu)$-contractivity of $P_{s/2}$ for the last inequality.

Let $\rho_{\alpha}(\cdot,\cdot)$ be the $\alpha$-resolvent kernel.
By the Chapman\textendash Kolmogorov equation,
\[
\begin{aligned}\int_{0}^{t}\Vert p_{s/2,x}-p_{s/2,y}\Vert_{L^{2}(\nu)}^{2}ds & \le e^{\alpha t}\int_{0}^{\infty}e^{-\alpha s}\Vert p_{s/2,x}-p_{s/2,y}\Vert_{L^{2}(\nu)}^{2}ds\\
 & =e^{\alpha t}\int_{0}^{\infty}e^{-\alpha s}[p(s,x,x)-2p(s,x,y)+p(s,y,y)]\,ds\\
 & =e^{\alpha t}[\rho_{\alpha}(x,x)-2\rho_{\alpha}(x,y)+\rho_{\alpha}(y,y)],
\end{aligned}
\]
which, together with Lemma \ref{lem:-2}, implies that
\[
\int_{0}^{t}\Vert p_{s/2,x}-p_{s/2,y}\Vert_{L^{2}(\nu)}^{2}ds\le C_{\alpha}R(x,y).
\]
Therefore, we deduce that
\begin{equation}
\int_{0}^{t}\Vert p_{s,x}-p_{s,y}\Vert_{L^{2}(\mu)}ds\le C_{T}R(x,y)^{1/2}.\label{eq:-40}
\end{equation}
Now the Hölder continuity \eqref{eq:-67} follows readily from \eqref{eq:-32}
and \eqref{eq:-40}.

\medskip{}

Next, we turn to the Hölder continuity of $u(t,x)$ in $t$. Let $t\ge0,\,\delta>0$.
By the definition of $u$,
\[
\begin{aligned}u(t+\delta,x)-u(t,x) & =\int_{t}^{t+\delta}P_{t+\delta-s}(g(s)\mu)(x)\,ds\\
 & \quad+\int_{0}^{t}\big[P_{t+\delta-s}(g(s)\mu)(x)-P_{t-s}(g(s)\mu)(x)\big]\,ds\\
\vphantom{\int_{t}^{t+\delta}} & =I_{1}(\delta)+I_{2}(\delta).
\end{aligned}
\]

For $I_{1}(\delta)$, in the same way as \eqref{eq:-32}, we have
\[
|I_{1}(\delta)|=\Big|\int_{0}^{\delta}P_{s}(g(t-s+\delta)\mu)(x)\,ds\Big|\le\Vert g\Vert_{L^{\infty}(0,T;L^{2}(\mu))}\int_{0}^{\delta}\Vert p_{s,x}\Vert_{L^{2}(\mu)}\,ds.
\]
By the Sobolev inequality \eqref{eq:-25},
\[
\Vert p_{s,x}\Vert_{L^{2}(\mu)}\le C\mathcal{E}(p_{s,x})^{(d_{s}-1)/2}\Vert p_{s,x}\Vert_{L^{2}(\nu)}^{2-d_{s}}.
\]
It follows from an argument similar to \eqref{eq:-75} that $\mathcal{E}(p_{s,x})\le s^{-1}\Vert p_{s/2,x}\Vert_{L^{2}(\nu)}^{2}$.
Therefore,
\[
\Vert p_{s,x}\Vert_{L^{2}(\mu)}\le Cs^{-(d_{s}-1)/2}\Vert p_{s/2,x}\Vert_{L^{2}(\nu)}^{d_{s}-1}\,\Vert p_{s,x}\Vert_{L^{2}(\nu)}^{2-d_{s}}.
\]
Using the Chapman\textendash Kolmogorov equation and the estimate
$p(t,x,y)\le C_{\ast}\,t^{-d_{s}/2}$, we deduce from the above inequality
that
\[
\begin{aligned}\int_{0}^{\delta}\Vert p_{s,x}\Vert_{L^{2}(\mu)}\,ds & \le C\int_{0}^{\delta}s^{-(d_{s}-1)/2}p(s,x,x)^{(d_{s}-1)/2}\,p(2s,x,x)^{1-d_{s}/2}\,ds\\
 & \le C_{\ast}\int_{0}^{\delta}s^{-3d_{s}/4+1/2}\,ds=C_{\ast}\,\delta^{\frac{3}{2}(1-d_{s}/2)}.
\end{aligned}
\]
Thus
\begin{equation}
|I_{1}(\delta)|\le C_{\ast}\Vert g\Vert_{L^{\infty}(0,T;L^{2}(\mu))}\delta^{\frac{3}{2}(1-d_{s}/2)}.\label{eq:-76}
\end{equation}

For $I_{2}(\delta)$, by the same argument as in the estimate of $|I_{1}(\delta)|$,
we have
\begin{equation}
\begin{aligned}|I_{2}(\delta)| & \le\Vert g\Vert_{L^{\infty}(0,T;L^{2}(\mu))}\int_{0}^{t}\Vert p_{s+\delta,x}-p_{s,x}\Vert_{L^{2}(\mu)}\,ds\\
 & \le C_{\ast}\Vert g\Vert_{L^{\infty}(0,T;L^{2}(\mu))}\int_{0}^{t}\Vert p_{s/2+\delta,x}-p_{s/2,x}\Vert_{L^{2}(\nu)}^{d_{s}-1}\,\Vert p_{s+\delta,x}-p_{s,x}\Vert_{L^{2}(\nu)}^{2-d_{s}}\,\frac{ds}{s^{(d_{s}-1)/2}}\\
 & \le C_{\ast}\Vert g\Vert_{L^{\infty}(0,T;L^{2}(\mu))}\int_{0}^{t}\Vert p_{s/2+\delta,x}-p_{s/2,x}\Vert_{L^{2}(\nu)}\,\frac{ds}{s^{(d_{s}-1)/2}}.
\end{aligned}
\label{eq:-77}
\end{equation}
For any $\theta\in[0,1]$, by the spectral representation,
\[
\begin{aligned}\Vert p_{s/2+\delta,x}-p_{s/2,x}\Vert_{L^{2}(\nu)}^{2} & =\Vert(P_{s/4+\delta}-P_{s/4})p_{s/4,x}\Vert_{L^{2}(\nu)}^{2}\\
 & =\int_{0}^{\infty}(1-e^{-\delta\lambda})^{2}e^{-s\lambda/2}\,d\Vert E_{\lambda}p_{s/4,x}\Vert_{L^{2}(\nu)}^{2}\\
 & \le\delta^{2\theta}\int_{0}^{\infty}\lambda^{2\theta}e^{-s\lambda/2}\,d\Vert E_{\lambda}p_{s/4,x}\Vert_{L^{2}(\nu)}^{2}\\
\vphantom{\int_{0}^{\infty}} & \le C_{\ast}\,(\delta/s)^{2\theta}\Vert p_{s/4,x}\Vert_{L^{2}(\nu)}^{2}=C_{\ast}\,(\delta/s)^{2\theta}\,p(s/2,x,x)\\
\vphantom{\int_{0}^{\infty}} & \le C_{\ast}\delta^{2\theta}s^{-2\theta-d_{s}/2},
\end{aligned}
\]
which, together with \eqref{eq:-77}, implies that
\[
|I_{2}(\delta)|\le C_{\ast}\Vert g\Vert_{L^{\infty}(0,T;L^{2}(\mu))}\delta^{\theta}\int_{0}^{t}s^{\frac{3}{2}(1-d_{s}/2)-\theta}\,\frac{ds}{s}.
\]
Therefore, for any $\theta<\frac{3}{2}(1-d_{s}/2)$, the estimate
\begin{equation}
|I_{2}(\delta)|\le C_{\theta}\Vert g\Vert_{L^{\infty}(0,T;L^{2}(\mu))}\delta^{\theta},\label{eq:-78}
\end{equation}
is valid. Combining \eqref{eq:-76} and \eqref{eq:-78}, we deduce
that
\begin{equation}
|u(t,x)-u(s,x)|\le C_{\theta}\Vert g\Vert_{L^{\infty}(0,T;L^{2}(\mu))}|t-s|^{\theta},\label{eq:-79}
\end{equation}
for all $0<\theta<\frac{3}{2}(1-d_{s}/2)$.

\medskip{}

Now the joint Hölder continuity \eqref{eq:-65} follows readily from
\eqref{eq:-67} and \eqref{eq:-79}.
\end{proof}
\begin{defn}
\label{def:}A function $u$ is called a \emph{weak solution }to the
PDE \eqref{eq:} if:

\smallskip{}

\noindent $\hypertarget{WS.1}{\text{(WS.1)}}$\enskip{}$u\in L^{2}(0,T;\mathcal{F}(\mathbb{S}\backslash\mathrm{V}_{0}))$
and $u$ has a weak derivative $\partial_{t}u$ in $L^{2}(0,T;\mathcal{F}^{-1}(\mathbb{S}))$;

\smallskip{}

\noindent $\hypertarget{WS.2}{\text{(WS.2)}}$\enskip{}For any $v\in\mathcal{F}(\mathbb{S}\backslash\mathrm{V}_{0})$,
\[
\langle\partial_{t}u(t),v\rangle_{\nu}=-\mathcal{E}\big(u(t),v\big)+\langle f(t,u(t),\nabla u(t)),v\rangle_{\mu},\;\;\text{a.e. }t\in[0,T];
\]
$\hypertarget{WS.3}{\text{(WS.3)}}$\enskip{}$\lim_{t\to0}u(t)=\psi$
in $L^{2}(\nu)$.
\end{defn}
\begin{rem}
\label{rem:-8}(i) The term $\langle f(t,u(t),\nabla u(t)),v\rangle_{\mu}$
in $\hyperlink{WS.2}{\text{(WS.2)}}$ is legitimate since $\nabla u$
is $\mu$-a.e. defined and $u\in\mathcal{F}(\mathbb{S})\subseteq C(\mathbb{S})$.

\ (ii) Notice that, in general, the equation \eqref{eq:} does not
admit a solution $u$ such that $u(t)\in\mathrm{Dom}(\mathcal{L})$
and $\partial_{t}u(t)\in L^{2}(\nu)$ for a.e. $t\in[0,T]$. Otherwise,
the functional $v\mapsto\langle f(t,u,\nabla u),v\rangle_{\mu}$ will
be $L^{2}(\nu)$-bounded, which contradicts with the singularity between
$\nu$ and $\mu$. Therefore, solutions to non-linear parabolic PDEs
on $\mathbb{S}$ can only have mild regularity in general. This is
a remarkable feature of non-linear PDEs on $\mathbb{S}$, which suggests
a significant distinction between the PDE theory on Euclidean spaces
and that on fractals.

(iii) We shall show that if $u$ is a weak solution to \eqref{eq:}
then $u\in C((0,T]\times\mathbb{S})$ (see Theorem \ref{thm:} below).
Therefore, Definition \ref{def:} coincides with the definition of
solutions in \citep[Definition 3.17]{LQ16}. The joint continuity
of solutions is needed for the validity of the Feynman\textendash Kac
representation given by \citep[Theorem 3.19]{LQ16}, which will be
crucial in the study of the Burgers equations on $\mathbb{S}$ (see
Section \ref{sec:-4}).
\end{rem}
\begin{prop}
\emph{\label{prop:}}Suppose that $g\in L^{2}(0,T;L^{2}(\mu))$. Then
the initial and boundary problem to the PDE
\begin{equation}
\left\{ \begin{aligned}\partial_{t}u & \,d\nu=\mathcal{L}u\,d\nu+g(t,x)d\mu,\;\;\text{in}\ (0,T]\times(\mathbb{S}\backslash\mathrm{V}_{0}),\\
u & =0\;\;\text{on}\ (0,T]\times\mathrm{V}_{0},\;\;u(0)=\psi
\end{aligned}
\right.\label{eq:-13}
\end{equation}
admits a unique weak solution $u$ given by
\[
u(t)=P_{t}\psi+\int_{0}^{t}P_{t-s}(g(s)\mu)ds,\;\;t\in[0,T].
\]
Moreover
\begin{equation}
\begin{aligned}\Vert u\Vert_{L^{\infty}(0,T;L^{2}(\nu))}+ & \Vert u\Vert_{L^{2}(0,T;\mathcal{F})}+\Vert\partial_{t}u\Vert_{L^{2}(0,T;\mathcal{F}^{-1})}\\
\vphantom{\int} & \le C_{\ast}\big(\Vert\psi\Vert_{L^{2}(\nu)}+\Vert g\Vert_{L^{2}(0,T;L^{2}(\mu))}\big).
\end{aligned}
\label{eq:-5}
\end{equation}
\end{prop}
\begin{proof}
Clearly, we only need to prove for the case when $\psi=0$. Let $u_{\delta}$
be the truncated convolution defined by \eqref{eq:-37}, and let $u$
be the convolution given by \eqref{eq:-35}. For any $v\in\mathcal{F}(\mathbb{S}\backslash\mathrm{V}_{0})$,
by \eqref{eq:-34},
\[
\langle\partial_{t}u_{\delta}(t),v\rangle_{\nu}=-\mathcal{E}(u_{\delta}(t),v)+\langle P_{\delta}(g(t-\delta)\mu),v\rangle_{\nu}=-\mathcal{E}(u_{\delta}(t),v)+\langle g(t-\delta),P_{\delta}v\rangle_{\mu},\;\;\text{a.e.}\;t\in(0,T].
\]
Since $\lim_{\delta\to0}P_{\delta}v=v$ uniformly, by considering
a subsequence if necessary and setting $\delta\to0$, we deduce that
\[
\langle\partial_{t}u(t),v\rangle_{\nu}=-\mathcal{E}(u(t),v)+\langle g(t),v\rangle_{\mu},\;\;\text{a.e.}\;t\in(0,T].
\]
Therefore, $u$ is a weak solution to \eqref{eq:-13}.

The estimate \eqref{eq:-5} follows readily from Lemma \ref{lem:-4},
and the uniqueness of solutions is an immediate consequence of \eqref{eq:-5}.
\end{proof}
We are now in a position to state and give the proof of the main result
of this section.
\begin{thm}
\label{thm:}Suppose that \eqref{eq:-1} and \eqref{eq:-2} hold.
Then \eqref{eq:} admits a unique weak solution $u$ satisfying the
following estimate
\begin{equation}
\begin{aligned}\Vert u\Vert_{L^{\infty}(0,T;L^{2}(\nu))} & +\Vert u\Vert_{L^{2}(0,T;\mathcal{F})}+\Vert\partial_{t}u\Vert_{L^{2}(0,T;\mathcal{F}^{-1})}\\
\vphantom{\int} & \le C_{K,T}\,\big(\Vert\psi\Vert_{L^{2}(\nu)}+\Vert f(\cdot,0,0)\Vert_{L^{2}(0,T;L^{2}(\mu))}\big),
\end{aligned}
\label{eq:-18}
\end{equation}
where $C_{K,T}>0$ is a constant depending only on $T$ and the Lipschitz
constant $K$ in \eqref{eq:-1}. Moreover, if $\tilde{u}$ is the
weak solution to \eqref{eq:} with initial value $\tilde{\psi}\in L^{2}(\nu)$,
then
\begin{equation}
\begin{aligned}\Vert u & -\tilde{u}\Vert_{L^{\infty}(0,T;L^{2}(\nu))}+\Vert u-\tilde{u}\Vert_{L^{2}(0,T;\mathcal{F})}\\
\vphantom{\int} & +\Vert\partial_{t}u-\partial_{t}\tilde{u}\Vert_{L^{2}(0,T;\mathcal{F}^{-1})}\le C_{K,T}\,\Vert\psi-\tilde{\psi}\Vert_{L^{2}(\nu)}.
\end{aligned}
\label{eq:-19}
\end{equation}
\medskip{}

Suppose, in addition, that $\psi\in\mathcal{F}(\mathbb{S}\backslash\mathrm{V}_{0})$
and $f(\cdot,0,0)=0$. Then
\begin{equation}
\Vert u\Vert_{L^{\infty}(0,T;\mathcal{F})}\le C_{K,T}\,\mathcal{E}(\psi)^{1/2}.\label{eq:-74}
\end{equation}
Moreover, $u(t,x)$ is jointly continuous in $(0,T]\times\mathbb{S}$,
with $\theta$-Hölder continuity in $t\in(0,T]$ for any $\theta<\frac{3}{2}(1-d_{s}/2)$
and $\frac{1}{2}$-Hölder continuity in $x\in\mathbb{S}$ with respect
to the resistance metric.
\end{thm}
\begin{proof}
We first prove the existence. Let $u^{0}(t)=P_{t}\psi,\;t\in[0,T]$.
By Proposition \ref{prop:}, we may define a sequence $\{u^{n}\}_{n\in\mathbb{N}_{+}}$
in $L^{2}(0,T;\allowbreak\mathcal{F}(\mathbb{S}\backslash\mathrm{V}_{0}))$
inductively by
\begin{equation}
\left\{ \begin{aligned}\partial_{t}u^{n}\;d\nu & =\mathcal{L}u^{n}\;d\nu+f^{n-1}(t)\;d\mu,\\
u^{n}\vert_{\mathrm{V}_{0}} & =0,\;\;u^{n}(0)=\psi,
\end{aligned}
\right.\label{eq:-8}
\end{equation}
where $f^{n-1}(t,x)=f(t,x,u^{n-1}(t,x),\nabla u^{n-1}(t,x))$. By
Proposition \ref{prop:}, $u^{n}\in L^{2}(0,T;\allowbreak\mathcal{F}(\mathbb{S}\backslash\mathrm{V}_{0}))$,
$\partial_{t}u^{n}\in L^{2}(0,T;\allowbreak\mathcal{F}^{-1}(\mathbb{S}))$.
Denote $w^{n}=u^{n}-u^{n-1},\,n\in\mathbb{N}_{+}$. By \eqref{eq:-8},
$w^{n+1},\,n\in\mathbb{N}_{+}$ is the solution to
\begin{equation}
\left\{ \begin{aligned}\partial_{t}w^{n+1} & \,d\nu=\mathcal{L}w^{n+1}\,d\nu+[f^{n}(t)-f^{n-1}(t)]\;d\mu,\\
w^{n+1} & \vert_{\mathrm{V}_{0}}=0,\;\;w^{n+1}(0)=0.
\end{aligned}
\right.\label{eq:-12}
\end{equation}
For any $\epsilon\in(0,1)$, by Lemma \ref{lem:}.(b), testing \eqref{eq:-12}
against $w^{n+1}(t)$ gives that
\[
\frac{1}{2}\frac{d}{dt}\Vert w^{n+1}(t)\Vert_{L^{2}(\nu)}^{2}\le-\mathcal{E}(w^{n+1}(t))+\frac{1}{32}\mathcal{E}(w^{n}(t))+\epsilon^{2}\Vert w^{n}(t)\Vert_{L^{2}(\mu)}^{2}+C_{K}\Vert w^{n+1}(t)\Vert_{L^{2}(\mu)}^{2}\big),
\]
where, and in the rest of the proof, $C_{K}>0$ denotes a generic
constant depending only on $K$ which may vary on different occasions.
Since $w^{n}|_{\mathrm{V}_{0}}=0$, by \eqref{eq:-14}, we have $\Vert w^{n}(t)\Vert_{L^{2}(\mu)}^{2}\le C_{\ast}^{2}\mathcal{E}(w^{n}(t))$.
Moreover,by Corollary \ref{cor:},
\[
\frac{1}{2}\frac{d}{dt}\Vert w^{n+1}(t)\Vert_{L^{2}(\nu)}^{2}\le-(1-C_{K}\,\epsilon^{2})\mathcal{E}(w^{n+1}(t))+\Big(\frac{1}{32}+C_{\ast}^{2}\epsilon^{2}\Big)\mathcal{E}(w^{n}(t))+C_{K}\,C_{\epsilon}^{2}\Vert w^{n+1}(t)\Vert_{L^{2}(\nu)}^{2},
\]
where $C_{\epsilon}>0$ is a constant depending only on $\epsilon$.
By choosing $\epsilon>0$ sufficiently small, we have that
\begin{equation}
\frac{d}{dt}\Vert w^{n+1}(t)\Vert_{L^{2}(\nu)}^{2}\le-\mathcal{E}(w^{n+1}(t))+\frac{1}{8}\mathcal{E}(w^{n}(t))+C_{K}\Vert w^{n+1}(t)\Vert_{L^{2}(\nu)}^{2},\label{eq:-17}
\end{equation}
By the above and Grönwall's inequality,
\begin{equation}
\Vert w^{n+1}(t)\Vert_{L^{2}(\nu)}^{2}+\int_{0}^{t}e^{C_{K}(t-s)}\mathcal{E}(w^{n+1}(s))\,ds\le\frac{1}{8}\int_{0}^{t}e^{C_{K}(t-s)}\mathcal{E}(w^{n}(s))\,ds,\label{eq:-69}
\end{equation}
which implies that
\begin{equation}
\begin{aligned} & \Vert u^{m}-u^{m-1}\Vert_{L^{\infty}(0,T;L^{2}(\nu))}+\bigg(\int_{0}^{T}e^{C_{K}(T-s)}\mathcal{E}(u^{m}(s)-u^{m-1}(s))\,ds\bigg)^{1/2}\\
 & \le2^{-m+n}\bigg[\Vert u^{n}-u^{n-1}\Vert_{L^{\infty}(0,T;L^{2}(\nu))}+\bigg(\int_{0}^{T}e^{C_{K}(T-s)}\mathcal{E}(u^{n}(s)-u^{n-1}(s))\,ds\bigg)^{1/2}\bigg],
\end{aligned}
\label{eq:-16}
\end{equation}
for all $m\ge n$, and that
\[
\begin{aligned}\Vert u^{n}\Vert_{L^{\infty}(0,T;L^{2}(\nu))} & +\bigg(\int_{0}^{T}e^{C_{K}(T-s)}\mathcal{E}(u^{n}(s))\,ds\bigg)^{1/2}\\
 & \le\Vert u^{1}\Vert_{L^{\infty}(0,T;L^{2}(\nu))}+\bigg(\int_{0}^{T}e^{C_{K}(T-s)}\mathcal{E}(u^{1}(s))\,ds\bigg)^{1/2}.
\end{aligned}
\]
Moreover, by Proposition \ref{prop:}, we have 
\[
\begin{aligned}\Vert u^{1}\Vert_{L^{\infty}(0,T;L^{2}(\nu))} & +\bigg(\int_{0}^{T}e^{C_{K}(T-s)}\mathcal{E}(u^{1}(s))\,ds\bigg)^{1/2}\\
\vphantom{\int_{0}^{T}} & \le C_{\ast}\,e^{C_{K}T}\,\big(\Vert\psi\Vert_{L^{2}(\nu)}+\Vert f(\cdot,u^{0},\nabla u^{0})\Vert_{L^{2}(0,T;L^{2}(\mu))}\big)\\
\vphantom{\int_{0}^{T}} & \le C_{\ast}\,e^{C_{K}T}\,\big(\Vert\psi\Vert_{L^{2}(\nu)}+\Vert u^{0}\Vert_{L^{2}(0,T;\mathcal{F})}+\Vert f(\cdot,0,0)\Vert_{L^{2}(0,T;L^{2}(\mu))}\big).
\end{aligned}
\]
Let $-\mathcal{L}=\int_{0}^{\infty}\lambda\;dE_{\lambda}$ be the
spectral decomposition. Then
\[
\begin{aligned}\Vert u^{0}\Vert_{L^{2}(0,T;\mathcal{F})}^{2} & \le T\Vert\psi\Vert_{L^{2}(\nu)}^{2}+\int_{0}^{T}\mathcal{E}(P_{t}\psi)\,dt\\
 & =T\Vert\psi\Vert_{L^{2}(\nu)}^{2}+\int_{0}^{T}\int_{0}^{\infty}\lambda e^{-2t\lambda}\,d\Vert E_{\lambda}\psi\Vert_{L^{2}(\nu)}^{2}\;dt\\
 & =T\Vert\psi\Vert_{L^{2}(\nu)}^{2}+\frac{1}{2}\int_{0}^{\infty}(1-e^{-2T\lambda})\,d\Vert E_{\lambda}\psi\Vert_{L^{2}(\nu)}^{2}\\
\vphantom{\int_{0}^{\infty}} & \le(T+1)\Vert\psi\Vert_{L^{2}(\nu)}^{2}.
\end{aligned}
\]
Therefore, we obtain that
\begin{equation}
\begin{aligned}\Vert u^{n}\Vert_{L^{\infty}(0,T;L^{2}(\nu))} & +\bigg(\int_{0}^{T}e^{C_{K}(T-s)}\mathcal{E}(u^{n}(s))\,ds\bigg)^{1/2}\\
\vphantom{\int_{0}^{T}} & \le C_{\ast}\,e^{C_{K}T}\,\big(\Vert\psi\Vert_{L^{2}(\nu)}+\Vert f(\cdot,0,0)\Vert_{L^{2}(0,T;L^{2}(\mu))}\big),\quad n\in\mathbb{N}_{+}.
\end{aligned}
\label{eq:-15}
\end{equation}

Furthermore, by \eqref{eq:-8}, $u^{m}-u^{n}$ is the solution to
\[
\left\{ \begin{aligned}\partial_{t}(u^{m}-u^{n}) & \;d\nu=\mathcal{L}(u^{m}-u^{n})\;d\nu+[f^{m-1}(t)-f^{n-1}(t)]\;d\mu,\\
(u^{m}-u^{n}) & \vert_{\mathrm{V}_{0}}=0,\;\;(u^{m}-u^{n})(0)=0.
\end{aligned}
\right.
\]
For any $v\in\mathcal{F}(\mathbb{S}\backslash\mathrm{V}_{0})$, by
the above equation,
\[
\big|\langle\partial_{t}(u^{m}-u^{n}),v\rangle_{\nu}\big|\le\big[\mathcal{E}(u^{m}-u^{n})^{1/2}+C_{K}\mathcal{E}(u^{m-1}-u^{n-1})^{1/2}\big]\mathcal{E}(v)^{1/2},
\]
which implies that
\begin{equation}
\Vert\partial_{t}(u^{m}-u^{n})\Vert_{\mathcal{F}^{-1}}\le\mathcal{E}(u^{m}-u^{n})^{1/2}+C_{K}\mathcal{E}(u^{m-1}-u^{n-1})^{1/2}.\label{eq:-20}
\end{equation}
By \eqref{eq:-16} and \eqref{eq:-15}, we see that 
\[
\Vert\partial_{t}(u^{m}-u^{n})\Vert_{L^{2}(0,T;\mathcal{F}^{-1})}\le C_{\ast}\,e^{C_{K}T}\,2^{-(m-n)}\,\big(\Vert\psi\Vert_{L^{2}(\nu)}+\Vert f(\cdot,0,0)\Vert_{L^{2}(0,T;L^{2}(\mu))}\big).
\]
Therefore, $\{u^{n}\}$ is a $\Vert\cdot\Vert_{\ast}\text{-}$Cauchy
sequence satisfying
\[
\Vert u^{n}\Vert_{\ast}\le C_{\ast}\,e^{C_{K}T}\,\big(\Vert\psi\Vert_{L^{2}(\nu)}+\Vert f(\cdot,0,0)\Vert_{L^{2}(0,T;L^{2}(\mu))}\big),\;\;n\in\mathbb{N}_{+},
\]
where
\[
\Vert u\Vert_{\ast}=\Vert u\Vert_{L^{\infty}(0,T;L^{2}(\nu))}+\Vert u\Vert_{L^{2}(0,T;\mathcal{F})}+\Vert\partial_{t}u\Vert_{L^{2}(0,T;\mathcal{F}^{-1})}.
\]
Therefore, there exists a $u\in L^{2}(0,T;\mathcal{F}(\mathbb{S}\backslash\mathrm{V}_{0}))$
such that $\lim_{n\to\infty}\Vert u^{n}-u\Vert_{\ast}=0$. It is clear
that $u$ is a weak solution to \eqref{eq:}, and the estimate \eqref{eq:-18}
holds as $\mathcal{E}^{1/2}(\cdot)$ and $\Vert\cdot\Vert_{\mathcal{F}}$
are equivalent on $\mathcal{F}(\mathbb{S}\backslash\mathrm{V}_{0})$.
This proves the existence.

\medskip{}

Suppose that $\tilde{u}$ is a weak solution to \eqref{eq:} with
initial value $\tilde{\psi}$. By an argument similar to \eqref{eq:-17}
and \eqref{eq:-20}, it can be shown that
\[
\frac{d}{dt}\Vert u(t)-\tilde{u}(t)\Vert_{L^{2}(\nu)}^{2}\le-\frac{1}{2}\mathcal{E}(u(t)-\tilde{u}(t))+C_{K}\Vert u(t)-\tilde{u}(t)\Vert_{L^{2}(\nu)}^{2},
\]
and that
\[
\Vert\partial_{t}(u-\tilde{u})\Vert_{\mathcal{F}^{-1}}\le C_{K}\,\mathcal{E}(u-\tilde{u})^{1/2}.
\]
The estimate \eqref{eq:-19} follows readily from the above two inequalities.
The uniqueness of solutions is now an immediate consequence of \eqref{eq:-19}.

\medskip{}

Suppose, in addition, that $\psi\in\mathcal{F}(\mathbb{S}\backslash\mathrm{V}_{0})$
and $f(\cdot,0,0)=0$. Then \eqref{eq:-69} also holds for $n=0$
with $u^{-1}=0$. Therefore,
\[
\int_{0}^{T}\mathcal{E}(u^{m}(t))\,dt\le e^{C_{K}T}\int_{0}^{T}\mathcal{E}(u^{0}(t))\,dt=e^{C_{K}T}\int_{0}^{T}\mathcal{E}(P_{t}\psi)\,dt\le Te^{C_{K}T}\mathcal{E}(\psi),
\]
which implies that
\begin{equation}
\int_{0}^{T}\mathcal{E}(u(t))\,dt\le Te^{C_{K}T}\mathcal{E}(\psi).\label{eq:-70}
\end{equation}

Now for any $\delta\in(0,T)$, $u$ is the solution to 
\[
\left\{ \begin{aligned}\partial_{t}u\,d\nu & =\mathcal{L}u\,d\nu+f(t,x,u,\nabla u)\,d\mu,\;\;t\in(t_{0},t_{0}+\delta],\\
u\vert_{\mathrm{V}_{0}} & =0,\;\;u\vert_{t=t_{0}}=u(t_{0}).
\end{aligned}
\right.
\]
Applying \eqref{eq:-70} to the above PDE and using $\Vert u\Vert_{L^{2}(\nu)}^{2}\le C_{\ast}\mathcal{E}(u)$
gives that
\begin{equation}
\frac{1}{\delta}\int_{t_{0}}^{t_{0}+\delta}\mathcal{E}(u(t))\,dt\le e^{C_{K}\delta}\,\mathcal{E}(u(t_{0})),\;\;\text{a.e.}\ t_{0}\in[0,T-\delta]\ \text{and any}\ \delta>0.\label{eq:-72}
\end{equation}
We claim that \eqref{eq:-72} implies \eqref{eq:-74}. We first show
the following lemma.

\medskip{}
\begin{lem*}
Let $h(t)$ be a locally integrable function on $[0,\infty)$ satisfying
\begin{equation}
\frac{1}{\delta}\int_{t}^{t+\delta}h(s)\,ds\le L\delta+h(t),\;\;\text{a.e.}\ t\in[0,\infty)\ \text{and any}\ \delta>0,\label{eq:-73}
\end{equation}
for some constant $L>0$. Then $h(t)-h(s)\le6L(t-s),\ \text{a.e.}\ 0<s\le t<\infty.$
\end{lem*}
To prove the lemma, suppose first that $h$ is differentiable on $(0,\infty)$.
Suppose the contrary that $h(t)-h(s)>3L(t-s)$ for some $0<s<t<\infty$.
Then there exists an $t_{0}\in(s,t)$ such that $h^{\prime}(t_{0})>3L$.
Moreover, $h(r)-h(t_{0})>3L(r-t_{0}),\;r\in[t_{0},t_{0}+\delta]$
for $\delta>0$ sufficiently small. This implies that $\frac{1}{\delta}\int_{t_{0}}^{t_{0}+\delta}h(r)\,dr>3L\delta/2+\,h(t_{0})$,
which contradicts \eqref{eq:-73}. This proves the lemma for differentiable
functions $h$.

For general $h$, let $h_{\epsilon}(t)=\frac{1}{\epsilon}\int_{t}^{t+\epsilon}h(s)\,ds,\;\epsilon>0$.
Then $h_{\epsilon}$ is differentiable and satisfies \eqref{eq:-73}
with $L$ replaced by $2L$. The above case gives that $h_{\epsilon}(t)-h_{\epsilon}(s)\le6L(t-s)$.
It remains to apply the Lebesgue differentiation theorem to complete
the proof of the lemma.

\bigskip{}

Now by \eqref{eq:-72} and Jensen's inequality, the function $h(t)=\log[\mathcal{E}(u(t))]$
satisfies \eqref{eq:-73} with $L=C_{K}$. It follows from the previous
lemma that 
\[
\mathcal{E}(u(t))\le\,e^{C_{K}(t-s)}\,\mathcal{E}(u(s)),\;\text{a.e.}\ 0<s\le t\le T.
\]
Using the above inequality and \eqref{eq:-70} again, we deduce that
\[
\mathcal{E}(u(t))\le\frac{1}{t}\int_{0}^{t}e^{C_{K}(t-s)}\,\mathcal{E}(u(s))\,ds\le e^{C_{K}t}\,\mathcal{E}(\psi),\;\;\text{a.e.}\ t\in(0,T],
\]
which implies \eqref{eq:-74}.

\medskip{}

We now prove the joint Hölder continuity. Let $g(t,x)=f(t,x,u(t,x),\nabla u(t,x))$.
Then $u$ is the solution to the PDE
\[
\partial_{t}u\;d\nu=\mathcal{L}u\,d\nu+g(t)d\mu.
\]
By Proposition \ref{prop:},
\[
u(t)=P_{t}\psi+\int_{0}^{t}P_{t-s}(g(s)\mu)\,ds.
\]
By \eqref{eq:-74} and the Sobolev inequality \eqref{eq:-25}, it
is easily seen that
\[
\Vert g\Vert_{L^{\infty}(0,T;L^{2}(\mu))}<\infty.
\]
We  now can apply Lemma \ref{lem:-1} and Proposition \ref{prop:}
and to deduce the desired joint Hölder continuity.
\end{proof}

\section{\label{sec:-4}The Burgers equations}

As an application of Theorem \ref{thm:} and the Feynman\textendash Kac
representation for (backward) parabolic PDEs on $\mathbb{S}$ in \citep[Theorem 3.19]{LQ16},
we study the initial-boundary value problem for the following analogue
on $\mathbb{S}$ of the Burgers equations on $\mathbb{R}$
\begin{equation}
\left\{ \begin{aligned}\partial_{t}u & \,d\nu=\mathcal{L}u\,d\nu+u\nabla u\,d\mu,\;\;\text{in}\ (0,T]\times(\mathbb{S}\backslash\mathrm{V}_{0}),\\
u & =0\ \;\text{on}\ (0,T]\times\mathrm{V}_{0},\;\;u(0)=\psi,
\end{aligned}
\right.\label{eq:-3}
\end{equation}
where $\psi\in\mathcal{F}(\mathbb{S}\backslash\mathrm{V}_{0})$. We
shall prove the existence and uniqueness of solutions to the equation
\eqref{eq:-3}, and derive the regularity of the solutions.
\begin{rem}
\label{rem:-11}We would like to point out a difference between the
Burgers equations on $\mathbb{S}$ and those on $\mathbb{R}$. The
Burgers equations on $\mathbb{R}$ can be exactly solved with an explicit
formula for the solutions via the Cole\textendash Hopf transformation,
and properties of solutions can be derived using the explicit formula.
However, this Cole\textendash Hopf type of transformation is not available
on $\mathbb{S}$. The Cole\textendash Hopf transformation reduces
the Burgers equation on $\mathbb{R}$ for $u$ to a heat equation
for $-\nabla(\log u)$. In contrast, on $\mathbb{S}$, the formal
expression $\mathcal{L}[\nabla(\log u)]$ is not well-defined, since
the gradient $\nabla(\log u)$ is only $\mu$-a.e. defined and therefore
$\nabla(\log u)\not\in\mathcal{F}(\mathbb{S})$ due to the singularity
between $\nu$ and $\mu$. Hence, different approaches must be employed
for the study of \eqref{eq:-3}.
\end{rem}
Let us start with the Feynman\textendash Kac representation for solutions
to parabolic PDEs on $\mathbb{S}$. Let $\{X_{t}\}_{t\ge0}$ and $\{W_{t}\}_{t\ge0}$
be Brownian motion and the representing martingale on $\mathbb{S}$
respectively, i.e. $\{X_{t}\}_{t\ge0}$ is the diffusion process associated
with the form $(\mathcal{E},\mathcal{F}(\mathbb{S}))$, and $\{W_{t}\}_{t\ge0}$
is the unique martingale additive functional having $\mu$ as its
energy measure such that $M_{t}^{[u]}=\int_{0}^{t}\nabla u(X_{r})\,dW_{r}$
for any $u\in\mathcal{F}(\mathbb{S})$, where $M^{[u]}$ is the martingale
part of $u(X_{t})-u(X_{0})$ (cf. \citep[Theorem 5.4]{Ku89} and \citep[Section 2]{LQ16}).
The following result was given in \citep[Theorem 3.19]{LQ16}, and
is an analogue on $\mathbb{S}$ of the representation theorem for
semi-linear PDEs on $\mathbb{R}^{d}$ established by S.~Peng in \citep{Pen91}.
See \citep[Section 3]{LQ16} for the definition of solutions to backward
stochastic differential equations (BSDEs) on $\mathbb{S}$.
\begin{thm}
\label{thm:-6}If the PDE \eqref{eq:} admits a weak solution $u$
jointly continuous in $(0,T]\times\mathbb{S}$, then
\[
(Y_{t},Z_{t})=(u(T-t,X_{t}),\nabla u(T-t,X_{t}))
\]
is the unique solution to the BSDE
\[
\left\{ \begin{aligned}dY_{t}= & -f(T-t,X_{t},Y_{t},Z_{t})d\langle W\rangle_{t}+Z_{t}dW_{t},\;\;t\in[0,\sigma^{(T)}),\\
Y_{\sigma^{(T)}} & =\Psi(\sigma^{(T)},X_{\sigma^{(T)}}),
\end{aligned}
\right.
\]
on $\big(\Omega,\mathbb{P}_{x}\big)$ for each $x\in\mathbb{S}$,
where $\sigma^{(T)}=T\wedge\inf\{t>0:X_{t}\in\mathrm{V}_{0}\}$, and
\[
\Psi(t,x)=\begin{cases}
\;\;0, & \mbox{if }(t,x)\in[0,T)\times\mathrm{V}_{0},\\
\psi(x), & \mbox{if }(t,x)\in\{T\}\times\mathbb{S}\backslash\mathrm{V}_{0}.
\end{cases}
\]
Moreover, the solution to \eqref{eq:} has the representation $u(T,x)=Y_{0}=\mathbb{E}_{x}(Y_{0})$
for all $x\in\mathbb{S}$.
\end{thm}
\begin{prop}
\label{prop:-2}The Burgers equation \eqref{eq:-3} admits a unique
weak solution $u$ satisfying the maximal principle below
\[
\Vert u\Vert_{L^{\infty}(0,T;L^{\infty})}\le\Vert\psi\Vert_{L^{\infty}}.
\]
Moreover,
\begin{equation}
\Vert u\Vert_{L^{\infty}(0,T;L^{2}(\nu))}+\Vert u\Vert_{L^{2}(0,T;\mathcal{F})}+\Vert\partial_{t}u\Vert_{L^{2}(0,T;\mathcal{F}^{-1})}\le C,\label{eq:-27}
\end{equation}
for some constant $C>0$ depending only on $\Vert\psi\Vert_{L^{\infty}}$
and $T$. The solution $u$ is jointly continuous in $(0,T]\times\mathbb{S}$,
with $\theta$-Hölder continuity in $t\in(0,T]$ for any $\theta<\frac{3}{2}(1-d_{s}/2)$
and $\frac{1}{2}$-Hölder continuity in $x\in\mathbb{S}$ with respect
to the resistance metric.
\end{prop}
\begin{proof}
Existence. We define the sequence $\{u^{n}\}_{n\in\mathbb{N}}\subseteq L^{2}(0,T;\mathcal{F})$
by induction as follows. Let $u^{0}(t)=P_{t}\psi$. Then $\Vert u^{0}\Vert_{L^{\infty}(0,T;L^{\infty})}\le\Vert\psi\Vert_{L^{\infty}}$.
Suppose that $u^{n-1}$ with $\Vert u^{n-1}\Vert_{L^{\infty}(0,T;L^{\infty})}\le\Vert\psi\Vert_{L^{\infty}}$
has been defined. The function $u^{n}$ is defined to be the unique
weak solution to the PDE (cf. Theorem \ref{thm:})
\[
\left\{ \begin{aligned}\partial_{t}u^{n} & d\nu=\mathcal{L}u^{n}\,d\nu+u^{n-1}\nabla u^{n}\,d\mu,\;\;\text{in}\ (0,T]\times(\mathbb{S}\backslash\mathrm{V}_{0}),\\
u^{n} & =0\;\;\text{on}\ (0,T]\times\mathrm{V}_{0},\;\;u^{n}(0)=\psi.
\end{aligned}
\right.
\]

To verify the definition of $\{u^{n}\}$, we must show that $\Vert u^{n}\Vert_{L^{\infty}(0,T;L^{\infty})}\le\Vert\psi\Vert_{L^{\infty}}$.
Without loss of generality, we only need to show that $\Vert u^{n}(T)\Vert_{L^{\infty}}\le\Vert\psi\Vert_{L^{\infty}}$.
By Theorem \ref{thm:-6}, $(Y_{t},Z_{t})=(u^{n}(T-t,X_{t}),\nabla u^{n}(T-t,X_{t}))$
is the unique solution to the BSDE
\begin{equation}
\left\{ \begin{aligned}dY_{t}= & -u^{n-1}(T-t,X_{t})\,Z_{t}\,d\langle W\rangle_{t}+Z_{t}\,dW_{t},\;\;t\in[0,\sigma^{(T)}),\\
Y_{\sigma^{(T)}} & =\Psi(\sigma^{(T)},X_{\sigma^{(T)}}),\hspace{2em}
\end{aligned}
\right.\label{eq:-28}
\end{equation}
where $\sigma^{(T)}=T\wedge\inf\{t>0:X_{t}\in\mathrm{V}_{0}\}$, and
\[
\Psi(t,x)=\begin{cases}
\;\;0, & \mbox{if }(t,x)\in[0,T)\times\mathrm{V}_{0},\\
\psi(x), & \mbox{if }(t,x)\in\{T\}\times\mathbb{S}\backslash\mathrm{V}_{0}.
\end{cases}
\]
For each $x\in\mathbb{S}\backslash\mathrm{V}_{0}$, we define a measure
$\tilde{\mathbb{P}}_{x}$ by
\[
\frac{d\tilde{\mathbb{P}}_{x}}{d\mathbb{P}_{x}}=\exp\bigg[\int_{0}^{\sigma^{(T)}}u^{n-1}(T-r,X_{r})\,dW_{r}-\frac{1}{2}\int_{0}^{\sigma^{(T)}}u^{n-1}(T-r,X_{r})^{2}\,d\langle W\rangle_{r}\bigg].
\]
The measure $\tilde{\mathbb{P}}_{x}$ is a probability measure. In
fact, by \citep[Corollary 4.3]{LQ16}, the quadratic process $\langle W\rangle$
is exponentially integrable, i.e.
\[
\sup_{x\in\mathbb{S}}\mathbb{E}_{x}[\exp(\beta\langle W\rangle_{T})]<\infty,
\]
for all $\beta,T>0$. Hence, in view of the uniform boundedness $\Vert u^{n-1}\Vert_{L^{\infty}(0,T;L^{\infty})}\le\Vert\psi\Vert_{L^{\infty}}$,
we see that the Novikov condition is satisfied and therefore $\tilde{\mathbb{P}}_{x}$
is a probability martingale measure. By \eqref{eq:-28},
\[
Y_{t}=Y_{0}+\int_{0}^{t}Z_{r}\,dW_{r}-\Big<\int Z_{r}\,dW_{r},\int u^{n-1}(T-r,X_{r})\,dW_{r}\Big>_{t}.
\]
Notice that
\[
\mathbb{E}_{\nu}\Big(\int_{0}^{T}Z_{r}^{2}\,d\langle W\rangle_{r}\Big)=\int_{0}^{T}\Vert\nabla u^{n}(T-r)\Vert_{L^{2}(\mu)}^{2}\,dr\le\Vert u^{n}\Vert_{L^{2}(0,T;\mathcal{F})}^{2}<\infty,
\]
which implies that $\mathbb{E}_{x}\big(\int_{0}^{T}Z_{r}^{2}\,d\langle W\rangle_{r}\big)<\infty$
for $\nu$-a.e. $x\in\mathbb{S}$ and therefore, for all $x\in\mathbb{S}$
in view of the quasi-continuity of the function $x\mapsto\mathbb{E}_{x}\big(\int_{0}^{T}Z_{r}^{2}\,d\langle W\rangle_{r}\big)$
and the fact that the empty set is the only subset of $\mathbb{S}$
having zero capacity since $\mathcal{F}(\mathbb{S})\subseteq C(\mathbb{S})$.
Hence, $\int Z_{r}dW_{r}$ is a $\mathbb{P}_{x}$-martingale for all
$x\in\mathbb{S}$. Moreover, it follows from the Girsanov theorem
that $\{Y_{t}\}_{t\ge0}$ is a $\tilde{\mathbb{P}}_{x}$-martingale,
and therefore,
\[
u^{n}(T,x)=Y_{0}=\tilde{\mathbb{E}}_{x}(Y_{0})=\tilde{\mathbb{E}}_{x}\big(Y_{\sigma^{(T)}}\big)=\tilde{\mathbb{E}}_{x}\big(\Psi(\sigma^{(T)},X_{\sigma^{(T)}})\big),
\]
which, together with the fact that $|\Psi|\le\Vert\psi\Vert_{L^{\infty}}$,
implies that $\Vert u^{n}(T)\Vert_{L^{\infty}}\le\Vert\psi\Vert_{L^{\infty}}$.
Hence, we conclude that $\Vert u^{n}\Vert_{L^{\infty}(0,T;L^{\infty})}\le\Vert\psi\Vert_{L^{\infty}}$,
and that the sequence $\{u^{n}\}$ is well-defined.

\medskip{}

Now, by Theorem \ref{thm:},
\[
\Vert u^{n}\Vert_{L^{2}(0,T;\mathcal{F})}+\Vert\partial_{t}u^{n}\Vert_{L^{2}(0,T;\mathcal{F}^{-1})}\le CT,\;\;n\in\mathbb{N},
\]
where $C>0$ is a generic constant depending only on $\Vert\psi\Vert_{L^{\infty}}$
which may vary on different occasions. Therefore, there exists a subsequence
$\{u^{n_{k}}\}$ and a $u\in L^{2}(0,T;\mathcal{F}(\mathbb{S}\backslash\mathrm{V}_{0}))$
such that $\partial_{t}u\in L^{2}(0,T;\mathcal{F}^{-1}(\mathbb{S}))$,
and
\begin{equation}
\lim_{k\to\infty}u^{n_{k}}=u,\;\;\text{weakly in }L^{2}(0,T;\mathcal{F}(\mathbb{S}\backslash\mathrm{V}_{0})),\label{eq:-4}
\end{equation}
\begin{equation}
\lim_{k\to\infty}\partial_{t}u^{n_{k}}=\partial_{t}u,\;\;\text{weakly in }L^{2}(0,T;\mathcal{F}^{-1}(\mathbb{S})).\label{eq:-7}
\end{equation}
Since $\Vert u^{n}\Vert_{L^{\infty}(0,T;L^{\infty})}\le\Vert\psi\Vert_{L^{\infty}}$,
the sequence $\{u^{n}\nabla u^{n}\}_{n\in\mathbb{N}_{+}}$ is bounded
in $L^{2}(0,T;L^{2}(\mu))$. By considering a subsequence of $\{u^{n_{k}}\}$
if necessary, we may assume that $\{u^{n_{k}}\nabla u^{n_{k}}\}$
is weakly convergent in $L^{2}(0,T;L^{2}(\mu))$. By the uniqueness
of weak limits,
\begin{equation}
\lim_{k\to\infty}u^{n_{k}}\nabla u^{n_{k}}=u\nabla u,\;\;\text{weakly in }L^{2}(0,T;L^{2}(\mu)).\label{eq:-24}
\end{equation}
Thus, it follows readily from \eqref{eq:-4}\textendash \eqref{eq:-24}
that $u$ is a weak solution to \eqref{eq:-3}. Moreover, the estimate
$\Vert u\Vert_{L^{\infty}(0,T;L^{\infty})}\le\Vert\psi\Vert_{L^{\infty}}$
follows as a corollary of the inequalities $\Vert u^{n}\Vert_{L^{\infty}(0,T;L^{\infty})}\le\Vert\psi\Vert_{L^{\infty}}$.

Testing \eqref{eq:-3} against $u(t)$ and using the Sobolev inequality
\eqref{eq:-25} gives that for any $\epsilon\in(0,1)$ and a.e. $t\in[0,T]$,
\[
\frac{d}{dt}\Vert u(t)\Vert_{L^{2}(\nu)}^{2}\le-\mathcal{E}(u(t))+\Vert\psi\Vert_{L^{\infty}}\,\big[\epsilon\mathcal{E}(u(t))^{1/2}+C_{\epsilon}\Vert u(t)\Vert_{L^{2}(\nu)}\big]\mathcal{E}(u(t))^{1/2}.
\]
Choosing $\epsilon>0$ sufficiently small gives that
\[
\frac{d}{dt}\Vert u(t)\Vert_{L^{2}(\nu)}^{2}\le-\frac{1}{2}\mathcal{E}(u(t))+C\Vert u(t)\Vert_{L^{2}(\nu)}^{2},\;\;\text{a.e.}\ t\in[0,T],
\]
from which the estimate \eqref{eq:-27} follows readily.

\medskip{}

Uniqueness. Suppose that $\bar{u}$ is also a weak solution to \eqref{eq:-3}.
Then $\Vert u\Vert_{L^{\infty}(0,T;L^{\infty})}+\Vert\bar{u}\Vert_{L^{\infty}(0,T;L^{\infty})}\le2\Vert\psi\Vert_{L^{\infty}}$.
For any $\epsilon>0$, testing the equation for $u(t)-\bar{u}(t)$
against $u(t)-\bar{u}(t)$ itself gives that
\[
\frac{d}{dt}\Vert u(t)-\bar{u}(t)\Vert_{L^{2}(\nu)}^{2}\le-\mathcal{E}(u(t)-\bar{u}(t))+C\Vert u(t)-\bar{u}(t)\Vert_{L^{2}(\mu)}\big[\mathcal{E}(u(t))^{1/2}+\mathcal{E}(u(t)-\bar{u}(t))^{1/2}\big],
\]
where, as before, $C>0$ is a generic constant depending only on $\Vert\psi\Vert_{L^{\infty}}$.
For any $\epsilon\in(0,1)$, using the Sobolev inequality \eqref{eq:-25},
we deduce that
\[
\frac{d}{dt}\Vert u(t)-\bar{u}(t)\Vert_{L^{2}(\nu)}^{2}\le\frac{C}{\epsilon}\Vert u(t)-\bar{u}(t)\Vert_{L^{2}(\nu)}^{2}+C(1+\epsilon)\,\big[\mathcal{E}(u(t))+\mathcal{E}(\bar{u}(t))\big].
\]
Therefore, 
\[
\Vert u(t)-\bar{u}(t)\Vert_{L^{2}(\nu)}^{2}\le C(1+\epsilon)\int_{0}^{t}e^{-C(t-s)/\epsilon}\big[\mathcal{E}(u(s))+\mathcal{E}(\bar{u}(s))\big]\,ds.
\]
By the dominated convergence theorem, setting $\epsilon\to0$ in the
above gives that $\Vert u(t)-\bar{u}(t)\Vert_{L^{2}(\nu)}=0,\;t\in[0,T]$,
which proves the uniqueness.

\medskip{}

We now turn to the proof of the joint Hölder continuity. Let $g(t)=u(t)\nabla u(t)$.
Then $|g(t)|\le\Vert\psi\Vert_{L^{\infty}}|\nabla u(t)|$. By an argument
similar to the proof of \eqref{eq:-74} in Theorem \ref{thm:}, we
may show that $\Vert g\Vert_{L^{\infty}(0,T;L^{2}(\mu))}<\infty$.
Since $u$ is the unique solution to
\[
\left\{ \begin{aligned}\partial_{t}u & \,d\nu=\mathcal{L}u\,d\nu+g(t)\,d\mu,\;\;\text{in}\ (0,T]\times(\mathbb{S}\backslash\mathrm{V}_{0}),\\
u & =0\ \;\text{on}\ (0,T]\times\mathrm{V}_{0},\;\;u(0)=\psi,
\end{aligned}
\right.
\]
we may now apply Lemma \ref{lem:-1} and Proposition \ref{prop:}
to obtain the desired joint Hölder continuity.
\end{proof}

\section*{Acknowledgements\addcontentsline{toc}{section}{Acknowledgements}}

The authors would like to thank Professor Ben Hambly for inspiring
discussions and providing with many invaluable suggestions on this
work. We also want to express our appreciation to Professor Masanori
Hino, Dr. Michael Hinz and Professor Alexander Teplyaev for their
helpful comments and discussions.

\noindent \renewcommand\bibname{References}


\noindent \begin{thebibliography}{33} \providecommand{\natexlab}[1]{#1} \providecommand{\url}[1]{\texttt{#1}} \expandafter\ifx\csname urlstyle\endcsname\relax   \providecommand{\doi}[1]{doi: #1}\else   \providecommand{\doi}{doi: \begingroup \urlstyle{rm}\Url}\fi
\bibitem[Barlow and Hambly(1997)]{BH97} M. Barlow and B. Hambly. \newblock Transition sensity estimates for {B}rownian motion on scale irregular   {S}ierpinski gaskets. \newblock In \emph{Annales de l'Institut Henri Poincare (B) Probability and   Statistics}, volume~33, pages 531--557. Elsevier, 1997.
\bibitem[Barlow and Kigami(1997)]{BK97} M. Barlow and J. Kigami. \newblock Localized eigenfunctions of the {L}aplacian on p.c.f self-similar   sets. \newblock \emph{Journal of the London Mathematical Society}, 56\penalty0   (02):\penalty0 320--332, 1997.
\bibitem[Barlow and Perkins(1988)]{BP88} M. Barlow and E. Perkins. \newblock Brownian motion on the {S}ierpinski gasket. \newblock \emph{Probability Theory and Related Fields}, 79\penalty0   (4):\penalty0 543--623, 1988.
\bibitem[Baudoin and Kelleher(2016)]{BK16} F. Baudoin and D. Kelleher. \newblock Differential forms on {D}irichlet spaces and {B}akry-\'{E}mery   estimates on metric graphs. \newblock \emph{Preprint arXiv:1604.02520}, 2016.
\bibitem[Caffarelli et~al.(1984)Caffarelli, Kohn, and Nirenberg]{CKN84} L. Caffarelli, R. Kohn, and L. Nirenberg. \newblock First order interpolation inequalities with weights. \newblock \emph{Compositio Mathematica}, 53\penalty0 (3):\penalty0 259--27,   1984.
\bibitem[Cipriani and Sauvageot(2003)]{CS03} F. Cipriani and J.-L. Sauvageot. \newblock Derivations as square roots of {D}irichlet forms. \newblock \emph{Journal of Functional Analysis}, 201\penalty0 (1):\penalty0   78--120, 2003.
\bibitem[Evans(2010)]{Eva10} L. Evans. \newblock \emph{Partial Differential Equations}, volume~19 of \emph{Graduate   Studies in Mathematics}. \newblock American Mathematical Society, 2nd edition, 2010.
\bibitem[Fitzsimmons et~al.(1994)Fitzsimmons, Hambly, and Kumagai]{FHK94} P. Fitzsimmons, B. Hambly, and T. Kumagai. \newblock Transition density estimates for {B}rownian motion on affine nested   fractals. \newblock \emph{Communications in Mathematical Physics}, 165\penalty0   (3):\penalty0 595--620, 1994.
\bibitem[Goldstein(1987)]{Gold87} S. Goldstein. \newblock Random walks and diffusions on fractals. \newblock In \emph{Percolation Theory and Ergodic Theory of Infinite Particle   Systems}, pages 121--129. Springer, 1987.
\bibitem[Grigor'yan et~al.(2003)Grigor'yan, Hu, and Lau]{GHL03} A. Grigor'yan, J. Hu, and K.-S. Lau. \newblock Heat kernels on metric measure spaces and an application to   semilinear elliptic equations. \newblock \emph{Transactions of the American Mathematical Society},   355\penalty0 (5):\penalty0 2065--2095, 2003.
\bibitem[Hambly(1992)]{Ham92} B. Hambly. \newblock {B}rownian motion on a homogeneous random fractal. \newblock \emph{Probability Theory and Related Fields}, 94\penalty0   (1):\penalty0 1--38, 1992.
\bibitem[Hambly(1997)]{Ham97} B. Hambly. \newblock {B}rownian motion on a random recursive {S}ierpinski gasket. \newblock \emph{The Annals of Probability}, 25\penalty0 (3):\penalty0   1059--1102, 1997.
\bibitem[Hino(2008)]{Hin08} M. Hino. \newblock Martingale dimensions for fractals. \newblock \emph{Institute of Mathematical Statistics}, 36\penalty0   (3):\penalty0 971--991, 2008.
\bibitem[Hino(2010)]{Hin10} M. Hino. \newblock Energy measures and indices of {D}irichlet forms, with applications   to derivatives on some fractals. \newblock \emph{Proceedings of the London Mathematical Society}, 100\penalty0   (1):\penalty0 269--302, 2010.
\bibitem[Hino(2013)]{Hin13} M. Hino. \newblock Upper estimate of martingale dimension for self-similar fractals. \newblock \emph{Probability Theory and Related Fields}, 156\penalty0   (3-4):\penalty0 739--793, 2013.
\bibitem[Hinz and Teplyaev(2015)]{HT15} M. Hinz and A. Teplyaev. \newblock Finite energy coordinates and vector analysis on fractals. \newblock \emph{Fractal Geometry and Stochastics V}, 70:\penalty0 209--227,   2015.
\bibitem[Hinz and Rogers(2015)]{HR15} M. Hinz and L. Rogers. \newblock Magnetic fields on resistance spaces. \newblock \emph{Preprint arXiv:1501.01100}, 2015.
\bibitem[Hinz and Teplyaev(2013)]{HT13} M. Hinz and A. Teplyaev. \newblock Dirac and magnetic {S}chr{\"o}dinger operators on fractals. \newblock \emph{Journal of Functional Analysis}, 265\penalty0 (11):\penalty0   2830--2854, 2013.
\bibitem[Hinz et~al.(2013)Hinz, R{\"o}ckner, and Teplyaev]{HRT13} M. Hinz, M. R{\"o}ckner, and A. Teplyaev. \newblock Vector analysis for {D}irichlet forms and quasilinear {PDE} and   {SPDE} on metric measure spaces. \newblock \emph{Stochastic Processes and Their Applications}, 123\penalty0   (12):\penalty0 4373--4406, 2013.
\bibitem[Ionescu et~al.(2012)Ionescu, Rogers, and Teplyaev]{IRT12} M. Ionescu, L. Rogers, and A. Teplyaev. \newblock Derivations, {D}irichlet forms and spectral analysis. \newblock \emph{Journal of Functional Analysis}, 263\penalty0 (8):\penalty0   2141--2169, 2012.
\bibitem[Kigami and Lapidus(1993)]{KL93} J.~Kigami and M.~L. Lapidus. \newblock Weyl's problem for the spectral distribution of {L}aplacians on   p.c.f. self-similar fractals. \newblock \emph{Communications in Mathematical Physics}, 158\penalty0   (1):\penalty0 93--125, 1993.
\bibitem[Kigami(1989)]{Ki89} J. Kigami. \newblock A harmonic calculus on the {S}ierpinski spaces. \newblock \emph{Japan Journal of Applied Mathematics}, 6\penalty0 (2):\penalty0   259--290, 1989.
\bibitem[Kigami(1993)]{Ki93} J. Kigami. \newblock Harmonic metric and {D}irichlet form on the {S}ierpinski gasket. \newblock \emph{Pitman Research Notes in Mathematics Series}, pages 201--218,   1993.
\bibitem[Kigami(2001)]{Ki01} J. Kigami. \newblock \emph{Analysis on Fractals}, volume 143 of \emph{Cambridge Tracts in   Mathematics}. \newblock Cambridge University Press, 2001.
\bibitem[Kusuoka(1989)]{Ku89} S.~Kusuoka. \newblock {D}irichlet forms on fractals and products of random matrices. \newblock \emph{Publications of the Research Institute for Mathematical   Sciences}, 25\penalty0 (4):\penalty0 659--680, 1989.
\bibitem[Kusuoka(1987)]{Ku87} S. Kusuoka. \newblock A diffusion process on a fractal. \newblock In Kiyosi It\^{o} and Nobuyuki Ikeda, editors, \emph{Probabilistic   Methods in Mathematical Physics: Proceedings of the Taniguchi International   Symposium (Katata and Kyoto, 1985)}, pages 251--274. Academic Press, 1987.
\bibitem[Liu and Qian(2017)]{LQ16} X. Liu and Z. Qian. \newblock Backward problems for stochastic differential equations on the   {S}ierpinski gasket. \newblock \emph{Stochastic Processes and Their Applications}, 128\penalty0 (10):\penalty0   3387--3418, 2018.
\bibitem[Mandelbrot(1977)]{Man77} B. Mandelbrot. \newblock \emph{Fractals: Form, Chance and Dimension}. \newblock San Francisco (CA, USA): W.~H.~Freeman \& Co. Ltd., 1977.
\bibitem[Peng(1991)]{Pen91} S. Peng. \newblock Probabilistic interpretation for systems of quasilinear parabolic   partial differential equations. \newblock \emph{Stochastics and Stochastic Reports}, 37\penalty0   (1-2):\penalty0 61--74, 1991.
\bibitem[Strichartz(2000)]{Str00} R. Strichartz. \newblock Taylor approximations on {S}ierpinski gasket type fractals. \newblock \emph{Journal of Functional Analysis}, 174\penalty0 (1):\penalty0   76--127, 2000.
\bibitem[Strichartz(2003)]{Str03} R. Strichartz. \newblock Function spaces on fractals. \newblock \emph{Journal of Functional Analysis}, 198\penalty0 (1):\penalty0   43--83, 2003.
\bibitem[Teplyaev(1998)]{Tep98} A. Teplyaev. \newblock Spectral analysis on infinite {S}ierpi{\'n}ski gaskets. \newblock \emph{Journal of Functional Analysis}, 159\penalty0 (2):\penalty0   537--567, 1998.
\bibitem[Teplyaev(2000)]{Tep00} A. Teplyaev. \newblock Gradients on fractals. \newblock \emph{Journal of Functional Analysis}, 174:\penalty0 128--154, 2000.
\end{thebibliography}
\end{document}